\documentclass[twoside, 11pt, leqno]{amsart}

\usepackage[T1]{fontenc}		
\usepackage[utf8]{inputenc}	
\usepackage{geometry}		

\usepackage{amsmath, amsmath, amsthm,amssymb}
\usepackage{graphicx,stackrel}
\usepackage{siunitx, booktabs, multirow} 
\usepackage{subcaption} 

\usepackage{enumerate} 

\usepackage{color}
\usepackage{hyperref}
\hypersetup{
    colorlinks=true,
    linktoc=all,
    linkcolor=blue,
}

\newcommand{\N}{\mathbb{N}}

\newcommand{\R}{\mathbb{R}}

\newcommand{\rd}{\mathbb{R}^d}
\newcommand{\Lip}{\textup{Lip}}

\newcommand{\vertiii}[1]{{\left\vert\kern-0.25ex\left\vert\kern-0.25ex\left\vert #1 \right\vert\kern-0.25ex\right\vert\kern-0.25ex\right\vert}}

\newtheorem{thm}{Theorem}[section]
\newtheorem{corollary}[thm]{Corollary}
\newtheorem{lemma}[thm]{Lemma}
\newtheorem{proposition}[thm]{Proposition}

\theoremstyle{remark}
\newtheorem{remark}[thm]{Remark}
\newtheorem{example}[thm]{Example}

\theoremstyle{definition}
\newtheorem{definition}[thm]{Definition}

\definecolor{cadmiumgreen}{rgb}{0.0, 0.42, 0.24}

\newcommand{\LL}{\mathcal{L}}

\newenvironment{description*}%
  {\begin{description}
    \setlength{\itemsep}{0.33em}
   \setlength{\mathindent}{1.5\leftmargin}
  }
  {\end{description}}
   
\makeatletter%
\let\orgdescriptionlabel\descriptionlabel
\renewcommand*{\descriptionlabel}[1]{%
  \let\orglabel\label
  \let\label\@gobble
  \phantomsection
  \edef\@currentlabel{#1}%
  \let\label\orglabel
  \orgdescriptionlabel{#1}%
}
\makeatother

\begin{document}

\title[Fractional Mean Field Games]{On fractional and nonlocal
  parabolic mean field games in the whole space}









\author[O. Ersland]{Olav Ersland}
\address{Department of Mathematical Sciences\\ Norwegian University of
  Science and Technology NTNU, Norway}
\curraddr{}
\email{olav.ersland@ntnu.no}

\author[E. R. Jakobsen]{Espen R. Jakobsen}
\address{Department of Mathematical Sciences\\ Norwegian University of
  Science and Technology NTNU, Norway}
\curraddr{}
\email{espen.jakobsen@ntnu.no}

\thanks{Both authors were supported by the Toppforsk (research
  excellence) project Waves and Nonlinear Phenomena (WaNP), grant
  no. 250070 from the Research Council of Norway.}



\begin{abstract}
 We study Mean Field Games (MFGs) driven by a large class of
nonlocal, fractional and anomalous diffusions in the whole space.
  These non-Gaussian diffusions are pure jump 
  L\'evy processes with some 
  $\sigma$-stable like behaviour. Included are $\sigma$-stable
  processes and fractional Laplace diffusion operators
  $(-\Delta)^{\frac{\sigma}2}$, tempered nonsymmetric processes 
   in Finance, spectrally one-sided processes,
   and sums of subelliptic operators of different orders.
  Our main results are existence and uniqueness 
  of classical solutions of MFG systems with nondegenerate diffusion
  operators of order $\sigma\in(1,2)$.
  We consider parabolic equations in the whole space with both
  local and nonlocal couplings. 
  Our proofs uses pure PDE-methods and build on ideas of Lions et
al.  The new ingredients are 
fractional heat kernel estimates, regularity results for fractional
Bellman, Fokker-Planck  and coupled Mean Field Game
equations, and a priori bounds and compactness of (very) weak
solutions of fractional Fokker-Planck equations in the whole
space. Our techniques requires no moment assumptions and uses a weaker
topology than Wasserstein. 
\end{abstract}

\maketitle
\tableofcontents

\section{Introduction}
  We study parabolic Mean Field Games (MFGs) driven by a large class of
  nonlocal, fractional and anomalous diffusions  in the whole space:
\begin{align}
  \begin{cases}
    \begin{aligned}
        -\partial_t u - \mathcal{L} u + H \left( x,u,Du \right) &= F \left( x, m \left( t \right) \right) && \text{ in } \left( 0,T \right) \times \R^d, \\[0.2cm]
        \partial_t m - \mathcal{L}^{*} m - div \left( m D_p H \left( x,u, Du  \right) \right) &= 0 && \text{ in } \left( 0,T \right) \times \R^d, \\[0.2cm]
      m \left( 0,x \right) = m_0 (x),\qquad u \left( x,T \right)& = G \left( x, m \left( T \right) \right),
      \end{aligned}
  \end{cases}
  \label{introMFGsystem}
\end{align}
where $H$ is a (nonlinear) Hamiltonian, 
$F$ and $G$ are source term and terminal condition, and $m_0$ 
an initial condition. 
Furthermore, $\mathcal{L}$ and its adjoint
$\mathcal{L}^{*}$, are non-degenerate fractional diffusion operators of order
$\sigma \in (1,2)$ of the form  
\begin{align}
    \mathcal{L} u (x) = \int_{\R^d} u (x+z) - u(x) - Du (x) \cdot z 1_{|z| <1}\ d \mu (z),
    \label{Levy_operator}
\end{align}
where  $\mu$ is a nonnegative Radon measure satisfying the L\'evy-condition
$\int_{\R^d} 1 \wedge |z|^2 \,d\mu \left( z \right) < \infty$, see
\ref{L1} and \ref{L2} below for precise assumptions.
The system is uniformly parabolic and consists of a
backward in time fractional Hamilton-Jacobi-Bellman (HJB) equation
coupled with a forward in time fractional Fokker-Planck (FP) equation.

\subsection*{Background}

MFGs is an emerging field of mathematics with a
wide and increasing range of applications in e.g. economy, network engineering,
biology, crowd and swarm control, and statistical learning
\cite{GLL10,GNP15}. It was introduced more or less at the same time by
Lasry and Lions \cite{LL06,lasry2007mean} and Caines, Huang and Malhamé
\cite{huang2006large}. Today there is a large and rapidly expanding
literature addressing a range of mathematical questions concerning
MFGs. We refer to the books and lecture notes
\cite{ACDPS20,CDI18,CDLL19,GPV16,BFY13} and references therein for an
overview of the theory and the current state of the art. 
Heuristically
a large number of
  identical players want to minimize some cost 
  depending on their own state and the distribution of the states of the
  other players, and the mean field game system arise as a characterisation of
  Nash equilibria when the number of players tends to infinity under
  certain symmetry assumptions. The optimal MFG feedback control is
  almost optimal also for finite player games with moderate to large
  numbers of players, and often provides the only practical way of
  solving also such games.

  In this paper the generic 
  player controls a stochastic
  differential equation (SDE) driven by  a
  pure jump Lévy process $L_t$ with characteristic triplet $(0,0,\mu)$
  \cite{applebaum2009levy},
\begin{align}\label{SDE}
      dX_t = \alpha_t \,dt + dL_t, 
\end{align}
with the aim of minimizing the cost functional
\begin{align*}
    \mathbb{E} \left[ \int_0^T \Big[L( X_s, \alpha_s ) + F ( X_s, m ( s ) ) \Big]ds + G \left( X_T, m \left( T \right) \right) \right]
\end{align*}
with respect to the control $\alpha_s$.
Here $L$ is the Legendre transform of $H$ with respect to the second
variable, $F$ and $G$ are running and terminal costs, and $m$ the 
distribution of the states of the other players.
If $u$ is the value function of the generic player, then formally the optimal
feedback control is $\alpha^*_t=-D_2H(x,Du)$ and $u$ satisfies the
HJB equation in \eqref{introMFGsystem}. The probability distribution
of the optimally controlled process $X_t^*$ then satisfies the FP
equation in \eqref{introMFGsystem}. Since the players are identical,
the distribution $m$ of all players will satisfy the same FP equation,
now starting from the initial distribution of
players $m_0$. This is a heuristic explanation for \eqref{introMFGsystem}. 

What differs from the standard MFG formulation is the type of noise
used in the model. In many real world applications, jump
processes or anomalous diffusions will better model the observed noise
than Gaussian processes \cite{MK00,CTbook,Wo01,applebaum2009levy}. One
example is symmetric $\sigma$-stable noise which correspond to fractional
Laplacian operators $\LL=(-\Delta)^{\frac{\sigma}2}$ for
$\sigma\in(0,2)$. In Finance the observed jump processes are not 
symmetric and $\sigma$-stable but rather 
non-symmetric and tempered.
An example is the one-dimensional CGMY process
\cite{CTbook} where $\frac{d\mu}{dz}(z)=\frac
C{|z|^{1+Y}}e^{-Gz^+-Mz^-}$ for $C,G,M>0$ and $Y\in(0,2)$.
Our assumptions cover a large class of uniformly elliptic operators
$\mathcal L$ that includes fractional Laplacians, generators of processes
used in Finance, anisotropic operators with different orders $\sigma$
in different directions, Riesz-Feller operators, and operators with
L\'evy measures that non-absolutley continuous, spectrally one-sided,
have no fractional moments, and a general behaviour at infinity.
We refer to Section
\ref{sec:fracHK} for a discussion, results, and examples. We also
analyse the system in the whole space, while many other papers focus on the
compact torus. For control problems and games, the whole space case is
usually more natural, but also more technical.

\medskip
\noindent{\bf Main results.}  Under structure and regularity assumptions
  on $\LL,H,F,G,m_0$, we show:
  \smallskip
  \begin{itemize}
    \item[(i)] {\em Existence of smooth solutions of
      \eqref{introMFGsystem}} with nonlocal and local coupling, see Theorems
      \ref{existence_theorem} and \ref{local_theorem}.
      \smallskip
    \item[(ii)] {\em Uniqueness of smooth solutions of
      \eqref{introMFGsystem}} with nonlocal and local coupling, see
      Theorems \ref{theorem:uniqueness} and \ref{local_uniqueness}.
  \end{itemize}
\smallskip

Our assumptions on $H,F,G$ are fairly standard
\cite{lionsCCDF,cardaliaguet2013mfg,ACDPS20} (except maybe that the
problem is posed on the whole space). For the existence results, we
note that the Hamiltonian $H$ (assumptions \ref{A3}--\ref{A5}) can be
both nonconvex and noncoercive. Since we consider nondegenerate
parabolic problems, the order of the equations have to be greater than one
and we do not need or impose semiconcavity assumptions.
 The proofs of the main results follow from an adaptation of the
 PDE-approach of Lions \cite{lionsCCDF,cardaliaguet2013mfg,ACDPS20},
 and existence is much more involved than uniqueness. Existence for
 MFGs with nonlocal coupling is proved using a Schauder fixed point argument and
   well-posedness, regularity, stability and compactness results for
    individual fractional HJB and fractional FP equations of the form:
\begin{align*}
    &\partial_t u - \mathcal{L} u + H \left( x,u, Du \right) = f \left(
    t,x \right), \\[0.2cm]
     & \partial_t m - \mathcal{L}^* m + div \left( b(t,x)m \right) =
      0. 
\end{align*}
Existence for MFGs with local coupling follows from an approximation
argument,  the results for nonlocal coupling, and regularity and
compactness results, in this case directly for the coupled MFG system.
\pagebreak

\noindent{\bf Secondary results:}
\smallskip
    \begin{itemize}
    \item[(iii)] {\em Fractional heat kernel estimates}, see Theorem
      \ref{L_heat_kernel_estimate} and Proposition \ref{kernel_estimates}. 
      \smallskip
    \item[(iv)] {\em Fractional HJB equations:} Regularity, existence,
      and space-time compactness of derivatives of classical solutions
      in Theorem \ref{thm_HJ} and Theorem \ref{HJ_Holder_thm}.
      \smallskip
    \item[(v)] {\em Fractional FP equations:} Well-posedness, 
      space-time compactness of derivatives, $C(0,T;P(\R^d))$ compactness, and
      global $L^\infty$ bounds of smooth solutions in Theorem
      \ref{fokker_planck_holder} (a), Theorem 
      \ref{fokker_planck_holder} (b) and (c), Proposition 
      \ref{prop:L1}, and Lemma \ref{L-infinity_estimate}.
  \end{itemize}
    \smallskip
    
  For both equations we show new high order regularity
  results of independent interest. These results are obtained from
  a Banach fixed point argument using semigroup/Duhamel
  representation of the solutions and bootstrapping in the spirit of
  \cite{droniou2003global,droniou2006fractal,imbert2005non}. Key  
  ingredients are very general fractional heat kernel estimates and
  global in time Lipschitz bounds for $u$ and $L^\infty$ bounds for
  $m$. The heat kernel estimates are based on
  \cite{grzywny2017estimates}, and we give some extensions, e.g. to operators
  with   general L\'evy measures at infinity and sums of subelliptic
  operators. To show space-time compactness of derivatives, we prove that they
  are space-time equi-continuous, combining uniform H\"older estimates
  in space with new time and mixed regularity estimates for the Duhamel
  representations of the solutions (see Section \ref{sec:fracHJB}).
  In the local coupling case, the HJB and FP equations have less
  regular data, and regularity can no longer be obtained through
  separate treatment of the equations. Instead we need to work directly
  on the coupled MFG system and apply a more refined bootstrapping
  argument based on fractional derivatives. These estimates also
  require better global in time Lipschitz and $L^\infty$ estimates
  the HJB and FP equations respectively. Here we use a variant of the
  Lipschitz bound of \cite{barles2012lipschitz} and provide a new
  $L^\infty$-estimate for the FP equation.
    
  For the Schauder fixed point argument to work and give existence for
  the MFG system, compactness in measure is needed 
  for a family of solutions of the FP-equation. We prove such
  compactness essentially through an analysis of very weak solutions
  of this equation: We prove preservation of positivity, mass, and
  $L^1$-norms, equicontinuity in time, and tightness. Our proof of
  equicontinuity is simple and direct, without probabilistic
  SDE-arguements as in
  e.g. \cite{cardaliaguet2013mfg,ACDPS20}. The tightness estimates
  are new in the fractional MFG setting and more challenging than in the
  local case. 

  This paper is the first to consider fractional MFGs in the whole
  space. To have compactness in measure on non-compact domains, a new
  ingredient is needed: tightness. Typically tightness is obtained
  through some moment condition on the familiy of measures. Such moment
  bounds depend both on the initial distribution and the generator of
  the process. In the local case when $L_t$ in \eqref{SDE} is a
  Brownian motion, then the process $X_t$ and FP solution $m$ have
  moments of any order, only limited by the number of moments  of
  $X_0$ and $m_0$. In the nonlocal/fractional case, $X_t$ and $m$ may
  have only limited (as for $\sigma$-stable processes) or even no
  fractional power moments at all, even when  $X_0$ and 
  $m_0$ have moments of all orders. We refer to Section \ref{sec:moment}
  for more examples, details, and discussion. Nonetheless it turns out
  that some generalized moment exists, and tightness and compactness
  can then be obtained. This relies on Proposition 
  \ref{prop:tail-control-function} (taken from
  \cite{Espen-Indra-Milosz-2020}), which gives the existence of a nice 
  ``Lyapunov'' function that can be integrated against $m_{0}$ and
  $\mu 1_{|z| \geq 1}$.

  {\em In this paper we prove tightness and compactness
  without any explicit moment conditions on the underlying processes $X_t$ or
  solutions of the FP equations $m$.} This seems to be new for MFGs
  even in the classical local case. Furthermore, $m$ is typically set in the
  Wasserstein-1 space $\mathcal W_1$ of measures with first moments,
  and compactness then requires more than one moment to be uniformly
  bounded.  Since our L\'evy processes and FP solutions may not have
  first moments, 
  we can not work in this setting. Rather we work in a weaker
  setting using a weaker Rubinstein-Kantorovich metric $d_0$
  (defined below) which is equivalent to weak convergence in measure
  (without moments). This is reflected both in the compactness and
  stability arguments we use as well as our assumptions on the
  nonlocal couplings.

  \subsection*{Literature}
In the case of Gaussian noise and local MGF systems, this type of
MFG problems with nonlocal or local coupling have been studied from
the start \cite{LL06,lasry2007mean,lionsCCDF,cardaliaguet2013mfg} and
today there is an extensive literature summarized e.g. in
\cite{ACDPS20,GPV16,BFY13} and references therein. For local MFGs
with local couplings, there are also results on weak solutions
\cite{lasry2007mean,P15,cardaliaguet2015second,ACDPS20}, a topic we do not
consider in this paper. Duhamel formulas have been used e.g. to prove
short-time existence and uniqueness in \cite{cirant2019time}. 

In the case of non-Gaussian noise and nonlocal MFGs or MFGs with
fractional diffusions, there is already some work. In
\cite{cesaroni2017stationary} the authors analyze a stationary MFG
system on the torus with fractional Laplace diffusions and both non-local 
and local couplings. 
Well-posedness of time-fractional MFG systems, i.e. systems with
fractional time-derivatives, are studied in
\cite{camilli2019time}. Fractional parabolic Bertrand
and Carnout MFGs are studied in the recent paper 
\cite{FIN2020}. These problems are posed in one space dimension, they
have a different and more complicated structure than ours, and the
principal terms are the (local) second derivative terms. The nonlocal
terms act as lower order perturbations. 
Moreover, during the rather long preparation of this paper we learned
that M. Cirant and A. Goffi were working on somewhat similar problems. Their
results have now been published in \cite{cirant2019existence}. 
They consider time-depending MFG systems on the torus with fractional
Laplace diffusions and nonlocal couplings. Since they assume
additional convexity and coercivity assumptions to ensure global in
time semiconcavity and Lipcshitz bounds on solutions, they consider
also fractional Laplacians of the full range of orders
$\sigma\in(0,2)$. Regularity results are given in terms of Bessel
potential and H\"older spaces, 
weak energy solutions are employed when $\sigma\in(0,1]$, and
  existence is obtained from the vanishing viscosity method.
Our setup is different in many ways, and more general in some (a large
class of diffusion operators, less smoothness on the data, problems
posed in the whole space, no moment conditions, fixed point
arguments), and most of our proofs and arguments are quite different from
those in \cite{cirant2019existence}. We also give results for local
couplings, which in view of the discussion above is a non-trivial extension.

\subsection*{Outline of paper}
This paper is organized as follows: In section \ref{sec:prelim} we
introduce notation, spaces, and give some preliminary assumptions and
results for the nonlocal operators. We state assumptions and give
existence and uniqueness results for MFG systems with nonlocal and
local coupling in Section \ref{sec:main}.  To prove these results, we
first establish 
fractional heat kernel estimates in Section \ref{sec:fracHK}.
Using these estimates and Duhamel representation formulas, we prove
regularity results for fractional Hamilton-Jacobi equations in Section
\ref{sec:fracHJB}.  
In Section \ref{sec:fracFP} we  establish results for fractional
Fokker-Planck equations, both regularity of classical solutions and
$C([0,T],P(\R^d))$ compactness.
In Sections \ref{sec:pf} and \ref{sec:loc} we prove the existence
result for nonlocal and local couplings respectively, while
uniqueness for nonlocal couplings is proved in Appendix
\ref{sec:pf_uniq}. Finally we prove a technical space-time regularity
lemma in Appendix \ref{pf-DH-lem}. 

\section{Preliminaries}\label{sec:prelim}

\subsection{Notation and spaces} By $C, K$ we mean various constants which may
change from line to line. The Euclidean norm on any
$\mathbb{R}^d$-type space is denoted by $|\cdot|$. For any subset $Q\subset
\mathbb{R}^N$ and for any bounded, possibly vector valued,
function on $Q$, we define the $L^{\infty}$ norms by
$\|w\|_{L^{\infty}(Q)} := \textup{ess sup}_{y\in Q} |w(y)|$. Whenever $Q= \rd$ or $Q=[0,T]\times \rd$, we denote $\|\cdot\|_{L^{\infty}(Q)} := \|\cdot\|_\infty$. Similarly, the norm in $L^p$ space is denoted by $\|\cdot\|_{L^{p}(Q)}$ or simply $\|\cdot\|_p$.
 We use $C_b(Q)$ and $UC (Q)$ to denote the
spaces of bounded continuous and uniformly continuous real valued
functions on $Q$, often we denote the norm $\|\cdot\|_{C_b}$ simply by
$\|\cdot\|_{\infty}$. Furthermore, $C_{b}^{k} ( \R^d )$ or
$C_{b}^{l,m} ( ( 0,T )\times \R^d )$ are subspaces of $C_b$ with $k$ bounded derivatives
or $m$ bounded space and $l$ bounded time derivatives.

By $P(\rd)$ we denote the set of Borel probability measure on
$\rd$. The Kantorovich-Rubinstein  distance $d_0(\mu_1,\mu_2)$ on the
space $P(\rd)$ is defined as  
\begin{align}\label{d0}
  d_0(\mu_1,\mu_2) := \sup_{f\in \Lip_{1,1}(\rd)}\Big\{\int_{\rd}f(x)
  d(\mu_1-\mu_2)(x)\Big\},
  \end{align}
where $\Lip_{1,1}(\rd) = \Big\{f : f \, \mbox{is Lipschitz continuous and} \, \|f\|_\infty, \|Df\|_\infty\leq 1 \Big\}$. 
Convergence in $d_0$ is equivalent to weak convergence of
measures (convergence in $(C_b)^*$), and hence tight subsets of
$(\mathbf{P},d_0)$ are precompact by Prokhorov's theorem.
We let the space $C([0,T];P(\R^d))$ be the set of $P(\R^d)$-valued
functions on $[0,T]$. It is a metric space with the metric $\sup_{t\in
  [0,T]}d_0(\mu(t),\nu(t))$, and tight equicontinuous subsets are
precompact by the Arzela-Ascoli and Prokhov theorems.

\subsection{Nonlocal operators}
Under the L\'evy condition
\begin{description}
    \item[(L1)\label{L1}] $\mu\geq 0$ is a Radon measure
      satisfying $\int_{\R^d} 1 \wedge |z|^2 \,d\mu \left( z \right) <
      \infty$, 
\end{description}
\smallskip
the operators $\LL$ defined in \eqref{Levy_operator} are in one to one
correspondence with the generators of pure jump L\'evy processes
\cite{applebaum2009levy}. One example is the symmetric $\sigma$-stable
processes and the fractional Laplacians,
$$-(-\Delta)^{\frac \sigma 2}\phi(x)=\int_{\R^d} \Big[\phi(x+z)-\phi(x)
-z\cdot D\phi(x) 1_{|z|<1}\Big] \frac{c_{d,\sigma}
  dz}{|z|^{d+\sigma}},\quad \sigma\in(0,2).$$
They are well-defined pointwise e.g. on
functions in $C_b\cap C^2$ by Taylor's theorem and Fubini:  
\begin{align*}
 |\LL \phi(x)| \leq \frac 12 \|D^2\phi\|_{C_b(B(x,1))} \int_{|z|<1}|z|^2d\mu(z)
 + 2\|\phi\|_{C_b}\int_{|z|\geq 1}d\mu(z)
  \quad \text{for}\quad x\in\R^d.
  \end{align*}
Let $\sigma\in[1,2)$. With more precise upper bounds on the integrals
  of $\mu$ near the origin: 
\begin{align}\label{u-bnd}
 \text{There is $c>0$ such that}\quad r^\sigma\int_{|z|<1} \frac{|z|^2}{r^2}\wedge 1\,
  d\mu(z)\leq c
  \quad \text{for}  \ r\in(0,1), 
  \end{align}
or equivalently,
  $r^{-2+\sigma}\int_{|z|<r} |z|^2 d\mu(z) + r^{-1+\sigma}\int_{r<|z|<1}
  |z| d\mu(z) + r^{\sigma}\int_{r<|z|<1} d\mu(z)\leq c$
for $r\in(0,1)$, we can have interpolation estimates for the operators
$\mathcal{L}$ in $L^p$.
\begin{lemma}    \label{L_p_bounds_eq}
  ($L^p$-bounds). Assume \ref{L1}, \eqref{u-bnd} with $\sigma\in[1,2)$, and $u\in C_b^2$. Then for all $p \in \left[ 1,\infty
    \right]$, and $r \in (0,1]$,
\begin{align}
    \| \mathcal{L} u \|_{L^{p} (\R^d)} \leq C \Big( \|D^2 u \|_{L^p} r^{2-\sigma} + \|Du\|_{L^p} \Gamma (\sigma,r) + \|u\|_{L^p}\mu(B_1^c) \Big)
\end{align}
where
\begin{align*}
    \Gamma (\sigma,r) = \begin{cases}
        |\ln r|, & \sigma  =1, \\
        r^{1-\sigma}-1, & 1< \sigma <2.
    \end{cases}
\end{align*}
    \label{L_p_bounds}
\end{lemma}
\begin{proof}
    For $ p \in  [ 1, \infty )$ we split $\mathcal{L} u$
    into three parts,
   $L_1 = \int_{B_r} u(x+z) - u(x) - Du (x) \cdot z \,d\mu (z)$, 
        $L_2 = \int_{B_1 \setminus B_r} u(x+z) - u(x)- Du (x) \cdot z \,d\mu (z)$,
    and $L_3 = \int_{\R^d \setminus B_1} u(x+z) - u(x) \,d\mu (z)$.
    Using Taylor expansions, Minkowski's integral inequality, and
    \eqref{u-bnd},
\begin{align*}
    \|L_1 \|_{L^p (\R^d)} &\leq \bigg( \int_{\R^d} |D^2 u (x)|^p \,dx
    \bigg)^{1/p} \int_{B_r} |z|^2 \,d\mu(z) 
    \leq C \|D^2 u\|_{L^p (\R^d)} r^{2-\sigma},\\
    \|L_2 \|_{L^p (\R^d)} &\leq 2\bigg( \int_{\R^d} |Du (x) |^p
    \,dx\bigg)^{1/p} \int_{B_1 \setminus B_r} |z|\,d\mu(z) 
    \leq C \|Du\|_{L^p (\R^d)} \Gamma (\sigma,r),\\
    \|L_3 \|_{L^p (\R^d)} &\leq 2 \bigg( \int_{\R^d} |u (x) |^p \,dx
    \bigg)^{1/p} \bigg(\int_{\R^d
      \setminus B_1} \bigg) \,d\mu (z)
    \leq 2 \|u\|_{L^p (\R^d) }\mu(B_1^c).
\end{align*}
Summing these estimates we obtain \eqref{L_p_bounds_eq}.  The case $p = \infty$ is similar, so we omit it. 
\end{proof}
Similar estimates are given e.g. in Section 2.5 in
\cite{GM:Book}. Note that assumption \eqref{u-bnd} holds for
$-(-\Delta)^{\beta/2}$ for any $\beta\in (0,\sigma]\setminus\{1\}$ and
    is related to the order of $\LL$.

\begin{remark}\label{fractional_laplace_estimate}
(a) When $\mu$ is symmetric, $\int_{B_1 \setminus B_r} Du (x) \cdot z
\,d\mu (z)=0$, $$\|L_2\|_{L^p}\leq
2\|u\|_p\int_{r<|z|<1}d\mu(z)\leq C\|u\|_pr^{-\sigma},$$
and
$\| \mathcal{L} u \|_{L^{p} (\R^d)} \leq C \big( \|D^2 u \|_{L^p}
r^{2-\sigma} + \|u\|_{L^p}r^{-\sigma} \big).$ Minimizing w.r.t. $r$
then yields
 \begin{align*}
        \| \mathcal L u \|_{L^p } \leq C\| D^2 u \|_p^{\sigma/2} \|u\|_p^{1- \sigma/2}.
    \end{align*}
This results holds for the fractional Laplacian
$\mathcal L=(-\Delta)^{\sigma/2}$ when $\sigma\in(1,2)$.
\medskip

\noindent (b)  When $\sigma\in(0,1)$, a similar argument shows that
$$ \|
\mathcal{L} u \|_{L^{p}} \leq C \big(\|Du\|_{L^p} r^{1-\sigma}
+ \|u\|_{L^p} r^{-\sigma} \big),$$
and we find that $\|
(-\Delta)^{\sigma/2} u \|_{L^p (\R^d)} \leq C\| D u \|_p^{\sigma}
\|u\|_p^{1- \sigma}$ for $\sigma\in(0,1)$.
\end{remark}

We define the adjoint of $\LL$ in the usual way.

 \begin{definition}
    (Adjoint). The adjoint of $\mathcal L$ is the operator
  $\mathcal{L}^*$ such that 
    \begin{align*}
        \langle \mathcal{L} f, g \rangle_{L^2 (\R^d) } = \langle  f,
        \mathcal{L}^* g \rangle_{L^2 (\R^d) }\qquad \text{for
          all}\qquad f,g \in C^2_c (\R^d).
    \end{align*}
\end{definition}
The $\mathcal L^*$ operator has the same form
as $\mathcal L$, with the ``antipodal'' L\'evy measure $\mu^*$:
\begin{lemma}    \label{adjoint_operator}
Assume \ref{L1} holds.  The adjoint operator $\mathcal{L}^*$ is given by
    \begin{align*}
        \mathcal{L}^* u (x) = \int_{\R^d} u (x+z) - u(x) - Du (x) \cdot z 1_{|z| <1}\ d \mu^* (z),
    \end{align*}
    where $\mu^* (B) = \mu (-B)$ for all Borel sets $B \subset \R^d$.
\end{lemma}
 This result is classical (see e.g. 
 Section 2.4 in \cite{GM:Book}). {\em Hence all assumptions and
 results in this paper for $\mu$ and $\LL$ automatically also holds
 for $\mu^*$ and  $\LL^*$ (and vice versa).}

 \subsection{Moments of L\'evy-measures, processes and FP
   equations}\label{sec:moment} 

Consider the solution $X_t$ of the SDE \eqref{SDE} (e.g. with
$X_0=x\in\R^d$) and the corresponding FP equation for its probability
distribution $m$, $m_t+ \text{div}(\alpha m) - \LL^*m=0$. If
$\alpha\in L^\infty$ and \ref{L1} 
holds, then it follows  that $X_t$ (and $m$) has $s>0$ moments if
and only if $\mu 1_{|z|>1}$ has $s$ moments \cite{applebaum2009levy}:
$$E|X_t|^s= \int_{\R^d} |x|^s m(dx,t)<\infty \quad \Longleftrightarrow\quad
\int_{|z|>1} |z|^sd\mu(z)<\infty.$$
The symmetric $\sigma$-stable processes have finite $s$-moments for
any $s\in(0,\sigma)$. It is well-known that smoothing properties of $\LL$ only
depend on the (moment) properties of $\mu 1_{|z|<1}$, and hence is
completely independent of the number of moments of $\mu 1_{|z|>1}$,
$X_t$ and $m(t)$. This fact is reflected in the elliticity assumption \ref{L2'}
in the next section, and follows e.g. from simple heat kernel
considerations in section \ref{sec:fracHK}, see Remark 
\ref{rem:smoothing}.

\medskip

\noindent{\em In this paper we will be as general as possible and assume no explicit
moment assumptions on $\mu 1_{|z|>1}$, $X_t$, and $m(t)$. The only condition
we impose on $\mu 1_{|z|>1}$ is \ref{L1}.}
\medskip

Note however, that we will still always have some sort of generalized
moments, but maybe not of power type, and these  ``moments'' will be
important for tightness and compactness for the FP equations. We refer
to section \ref{sec:fracFP} and Proposition
\ref{prop:tail-control-function} for more details.

\section{Existence and uniqueness for fractional MFG systems}
\label{sec:main}

Here we state our assumptions and the existence and uniqueness results
for classical solutions of the system \eqref{introMFGsystem} both with nonlocal
 and local couplings.

\subsection{Assumptions on the fractional operator $\mathcal L$ in (\ref{Levy_operator})}\label{assumpL}
We assume \ref{L1} and
\smallskip
\begin{description}
      \item[(L2')\label{L2'}] (Uniform ellipticity) There are constants $\sigma \in (1,2)$ and $C >0$ such that
        \begin{align*}
            \frac{1}{C} \frac{1}{|z|^{d+\sigma}} \leq \frac{d\mu}{dz} \leq C \frac{1}{|z|^{d+\sigma}}\quad \text{for}\quad |z| \leq 1.
        \end{align*}
\end{description}
        \smallskip

These assumptions are satisfied by generators $\mathcal L$ of
pure jump processes whose infinite activity part is close
to $\alpha$-stable. But scale invariance is not required
nor any restrictions on the tail of $\mu$ except for \ref{L1}. Some
examples are $\alpha$-stable processes, tempered 
$\alpha$-stable processes, and the nonsymmertic CGMY process in Finance
\cite{CTbook,applebaum2009levy}. Note that the
upper bound on $\frac{d\mu}{dz}$ implies that \eqref{u-bnd} holds.
A much more general condition than \ref{L2'} is:  
\smallskip
\begin{description}
    \item[(L2)\label{L2}]  There is $\sigma \in ( 1,2 )$, such that \smallskip 

       \noindent  (i) $\mu$ satisfies the upper bound \eqref{u-bnd}.\medskip

      \noindent  (ii) There is $\mathcal K
>0$ such that the heat kernels $K_\sigma$ and $K_\sigma^*$ of $\mathcal
L$ and $\mathcal L^*$ satisfy for $K=K_\sigma,K_\sigma^*$\ :
$K\geq0$, $\|K(t,\cdot)\|_{L^1(\R^d)}=1$, and 
        \begin{align*}
            \|D^{\beta} K
  (t,\cdot) \|_{L^p (\R^d)} \leq \mathcal K t^{-\frac{1}{\sigma}\big(|\beta|+(1-\frac1p)d\big)}\quad
            \text{for $t\in(0,T)$}
        \end{align*}
and any $p\in[1,\infty)$ and multi-index $\beta\in \N_0^{d}$ where $D$ is the
gradient in $\R^d$.
\end{description}
\smallskip
The heat kernel is a transition probability/fundamental
solution. Under \ref{L2} L\'evy measures need not be 
absolutely continuous, e.g. 
$\mathcal L=-\Big(\!-\frac{\partial^2}{\partial
  x_1^2}\Big)^{\sigma_1/2}-\dots-\Big(\!-\frac{\partial^2}{\partial 
      x_d^2}\Big)^{\sigma_d/2}$ for $\sigma_1,\dots,\sigma_d\in(1,2)$
satisfies \ref{L2} with $\sigma=\min_i\sigma_i$
 and $d\mu(z)=\sum_{i=1}^d\frac{dz_i}{|z_i|^{1+\sigma_i}}\Pi_{j\neq
  i}\delta_0(dz_j)$.  
See Section \ref{sec:fracHK} for precise definitions, a
proof that \ref{L2'} implies \ref{L2}, more examples and extensions.        

In the local coupling case, we need in addtion to \ref{L2} also the
following assumption:
\smallskip
\begin{description}
    \item[(L3)\label{L3}] 
Let the cone $\mathcal{C}_{\eta,r}(a) : = \{z \in B_r: (1 -
\eta)|z||a| \leq |\langle a,z \rangle| \}$. There is $\beta \in (0,2)$ such that for every $a \in \R^d$ there exist  $0 < \eta < 1$ and $C_\nu>0$, and for all $r >0$,
$$\int_{\mathcal{C}_{\eta,r}(a)}|z|^2 \nu(dz) \geq C_\nu \eta^{\frac{d-1}{2}}r^{2-\beta}.$$
\end{description}
This assumption is introduced in \cite{barles2012lipschitz} to prove
Lipschitz bounds for fractional HJB equations. It holds e.g. for fractional
Laplacians \cite[Example 1]{barles2012lipschitz} and then also if the
inqualities of \ref{L2'} holds for all $z\in\R^d$. Since the
assumption is in integral form, it also holds for
non-absolutely continuous L\'evy measures, spectrally one-sided
processes, sums of operators etc.

\medskip
\subsection{Fractional MFGs with nonlocal coupling}
We consider the MFG system
\smallskip
\begin{align}
  \begin{split}
    \left\{ \begin{array}{ll}
        -\partial_t u - \mathcal{L} u + H \left( x,u,Du \right) = F \left( x, m \left( t \right) \right) & \text{in}\quad \left( 0,T \right) \times \R^d, \\[0.2cm]
      \partial_t m - \mathcal{L}^* m - div \left( m D_p H \left( x,u, Du  \right) \right) = 0 & \text{in}\quad \left( 0,T \right) \times \R^d, \\[0.2cm]
      m \left(x, 0 \right) = m_0(x) ,\qquad u \left( x,T \right) = G
      \left( x, m \left( T \right) \right) &\text{in}\quad \R^d,
    \end{array}
    \right.
  \end{split}
  \label{frac_MFG_system}
\end{align}
where the functions $F,G: \R^d \times P \left( \R^d \right) \to \R$ are
non-local coupling functions, and $H: \R^d \times \R \times \R^d \to \R$ is the
Hamiltonian. We impose fairly standard assumptions on the data and
nonlinearities \cite{lionsCCDF,cardaliaguet2013mfg,ACDPS20} (but note
we use the metric $d_0$ and not Wasserstein-1):
\medskip
\begin{description}
    \item[(A1)\label{A1}] There exists a $C_0 > 0$ such that for all $\left( x_1, m_1 \right), \left( x_2, m_2 \right) \in \R^d \times P \left( \R^d \right)$:
        \begin{align*}
 \qquad\qquad    | F ( x_1, m_1 ) - F ( x_2, m_2 ) | +| G ( x_1, m_1 )
 - G ( x_2, m_2 ) |&\leq C_0 (|x_1 - x_2 | + d_0 ( m_1, m_2)).
        \end{align*}
      \item[(A2)\label{A2}]  There exist constants $C_F,C_G > 0$, such that
        $$\sup_{m \in P \left( \R^d \right)} \| F \left( \cdot, m \right)\|_{C_b^{2} \left( \R^d \right) } \leq C_F\quad\text{and}\quad\sup_{m \in P \left( \R^d \right)}\| G \left( \cdot, m \right)\|_{W^{3,\infty} \left( \R^d \right)} \leq C_G.$$
  
   \item[(A3)\label{A3}] For every $R > 0$ there is $C_R >0$ such that for $x \in \R^d , u \in \left[ -R,R \right] , p \in B_R$, $\alpha \in \N_0^N$, $|\alpha| \leq 3$,
      \begin{align*}
        |D^{\alpha} H \left( x,u,p \right) | \leq C_R.
      \end{align*}

    \item[(A4)\label{A4}] For every $R > 0$ there is $C_R >0$ such that for $x,y \in \R^d, u \in \left[ -R,R \right], p \in \R^d$:
      \begin{align*}
        |H \left( x,u,p \right) - H \left( y,u,p \right)| \leq C_R \left( |p|+1 \right) |x-y|. 
      \end{align*}
    \item[(A5)\label{A5}] There exists $\gamma \in \R$ such that for all $x \in \R^d, u,v \in \R, u \leq v, p \in \R^d$,
      \begin{align*}
        H \left( x,v,p \right) - H \left( x,u,p \right) \geq \gamma \left( v-u \right).
      \end{align*}
  \item[(A6)\label{A6}]  
    $m_0 \in W^{2,\infty} \left( \R^d \right) \cap \mathbf{P} ( \R^d)$.
\end{description}

Note that convexity or coercivity of $H$ is not assumed at this point
and that we identify probability measures and their density functions
(see \ref{A6}). 
\begin{definition}
     (Classical solution) A classical solution of
  \eqref{frac_MFG_system} is a pair $(u,m)$ such that (i) $u,m\in
  C(\R^d\times[0,T])$, (ii) $m\in C([0,T]; P ( \R^{d} ))$, (iii) 
  $Du,D^2u, \mathcal L u,u_t, Dm, \mathcal L^* m, m_t\in
  C(\R^d\times(0,T))$, and (iv) $( u,m )$ solves
  \eqref{frac_MFG_system} at every point.
 \end{definition}
\begin{thm}
  (Existence of classical solutions) Assume \ref{L1},\ref{L2},
  \ref{A1}--\ref{A6}. Then there exists a classical solution $\left(
  u,m \right)$ of (\ref{frac_MFG_system}) such that $u \in C_{b}^{1,3} ( ( 0,T ) \times \R^d )$ and $m \in C_{b}^{1,2} ( ( 0,T ) \times \R^d ) \cap C( [0,T]; P ( \R^{d} ) )$.
  \label{existence_theorem}
\end{thm}
 The proof will given in Section \ref{sec:pf}. It is an adaptation of
 the fixed point argument of P.-L. Lions
 \cite{lionsCCDF,cardaliaguet2013mfg,ACDPS20} and requires a series of
 a priori, regularity, and compactness estimates for fractional HJB and
fractional FP equations given in  Sections \ref{sec:fracHJB} and
\ref{sec:fracFP}.

For uniqueness, we add the following assumptions:
\begin{description}
    \item[(A7)\label{A7}] $F$ and $G$ satisfy monotonicity conditions: 
\begin{align*}
    \int_{\R^d} \left( F \left( x, m_1 \right) - F \left( x, m_2 \right) \right) d \left( m_1 -m_2 \right) \left( x \right) &\geq 0 \qquad \forall m_1,m_2 \in P ( \R^d ), \\[0.2cm]
  \int_{\R^d} \left( G \left( x, m_1 \right) - G \left( x, m_2 \right) \right) d \left( m_1 -m_2 \right) \left( x \right) &\geq 0 \qquad \forall m_1,m_2 \in P ( \R^d ).
\end{align*}
    \item[(A8)\label{A8}] The Hamiltonian $H = H \left( x,p \right)$ and is uniformly convex with respect to $p$:
\begin{align*}
  \exists C > 0, \ \ \frac{1}{C} I_d \leq D_{pp}^2 H \left( x,p \right) \leq C I_d.
\end{align*}
\end{description}

\begin{thm}
    Assume \ref{L1}, \ref{A1}-\ref{A8}. Then there is at most one classical solution of the MFG system (\ref{frac_MFG_system}).
  \label{theorem:uniqueness}
\end{thm}
Since $\mathcal{L}$ and $\mathcal{L}^*$ are ajoint operators, the
proof of uniqueness is essentially the same as the proof in the
College de France lectures of P.-L. Lions
\cite{lionsCCDF,cardaliaguet2013mfg,ACDPS20}. For the readers convenience we
give the proof in Appendix \ref{sec:pf_uniq}.

\begin{example}\label{ex1}
  (a)  $F ( x,m ) = ( \rho \ast m ) (x)$ satisfies \ref{A1} and \ref{A2}
if $\rho \in C_{b}^{2} ( \R^d )$. 
    \medskip
    
\noindent (b) $F ( x,m ) = \int_{\R^d} \Phi ( z, ( \rho \ast m ) ( z )
) \rho ( x-z ) dz$ satisfies \ref{A1} and \ref{A2} if  $\rho\in
C^2_b$ and $\Phi\in C^1$.

    \medskip
    
\noindent (c) Both functions satisfy \ref{A7} if
$\rho\geq0$ and $\Phi$ is nondecreasing in its second argument.
\end{example}

\subsection{Fractional MFG with Local Coupling}
We consider the MFG system
\begin{align}
  \begin{split}
    \left\{ \begin{array}{ll}
        -\partial_t u - \mathcal{L} u + H \left( x, Du \right) = f \left( x, m \left( t,x \right) \right) & \text{ in } \left( 0,T \right) \times \R^d \smallskip \\ 
      \partial_t m - \mathcal{L}^* m - div \left( m D_p H \left( x, Du  \right) \right) = 0 & \text{ in } \left( 0,T \right) \times \R^d \smallskip \\ 
      m \left( 0 \right) = m_0 , u \left( x,T \right) = g \left( x \right),
    \end{array}
    \right.
  \end{split}
  \label{loc_MFG_system}
\end{align}
where the coupling term $f$ are local and only depend on the value of
$m$ at $\left( x,t \right)$. Again we impose fairly standard assumptions
on $f$,  $g$ and  $H$ \cite{lionsCCDF,cardaliaguet2013mfg}: \medskip 
\begin{description}
    \item[(A2')\label{A2'}] (Regularity) $f \in C^{2} ( \R^d \times [ 0, +\infty ) )$ 
        with $\| f ( \cdot, k )  \|_{C_{b}^2} \leq C_{k} $, 
        and $g \in  C_{b}^{3} ( \R^d  )$.
        \medskip
    \item[(A2'')\label{A2''}] (Uniform bound $f$) $\| f  \|_{C_{b}} \leq K_{f}$ for $K_{f} > 0$. 
        \medskip
    \item[(A3')\label{A3'}] (Lipschitz bound $H$)  $\| D_{p} H  \|_\infty \leq L_{H} $ for $L_{H}> 0$. \medskip
\end{description}

\begin{thm}
    Assume \ref{L1}--\ref{L3}, \ref{A3}-\ref{A6}, \ref{A2'}, and either \ref{A2''} or \ref{A3'}. 
    Then there exists a classical solution
    $\left( u,m \right)$ of (\ref{loc_MFG_system}) such that $u \in C_{b}^{1,3} ( ( 0,T ) \times \R^d )$ and $m \in C_{b}^{1,2} ( ( 0,T ) \times \R^d ) \cap C( [0,T]; P ( \R^{d} ) )$.
    \label{local_theorem}
\end{thm}
The proof of this result is given in Section \ref{sec:loc}.
The idea is to approximate by a MFG system with nonlocal coupling
and use the compactness and stability results to pass to the
limit. These results rely on new regularity reults. As opposed to
the case of nonlocal coupling, it not enough to consider the HJB and
FP equations separately, in this local coupling case, regularity has
to be obtained directly for the coupled system. 
This requires the use of fractional regularity and bootstrap arguments.

For uniqueness we follow \cite{lionsCCDF,cardaliaguet2013mfg} and look at
the more general MFG system 
\begin{align}
    \label{second_order_local_coupling_uniqueness}
  \left\{
    \begin{array}{ll}
        - \partial_t u - \mathcal{L} u + H \left( x, Du,m \right) = 0 & \text{in } \R^d \times \left( 0,T \right) \\
	\partial_t m - \mathcal{L}^* m - div \left(m D_p H \left(x, Du \left(t, x \right),m \right) \right) = 0 & \text{in }   \R^d \times \left( 0,T \right)\\
	m \left(0 \right) = m_0 \ , \ u \left(x,T \right) = G \left(x \right), & 
    \end{array}
    \right.
\end{align}
where $H=H \left( x,p,m \right)$ is convex in $p$ and
\begin{description}
        \bigskip
    \item[(A10)\label{A10}] $\begin{bmatrix}
	  m \partial_{pp}^2 H & \frac{1}{2} m \partial_{pm}^2 H \\
	  \frac{1}{2} m \left( \partial_{pm}^2 H \right)^T & - \partial_{m} H
  \end{bmatrix} > 0 \ \ \text{ for all } \left( x,p,m \right) \text{ with } m>0.$
  \medskip
\end{description}
Note that when $H \left( x,p,m \right) = \tilde{H} \left( x,p \right) - F \left( x,m \right)$, we recover assumption \ref{A8}.

\begin{thm}
  Assume \ref{L1}, \ref{A10}, and $H = H \left( x,p,m \right) \in
  C^2$. Then (\ref{loc_MFG_system}) has at most one classical solution. 
  \label{local_uniqueness}
\end{thm}

We skip the proof which in view of adjointness of $\mathcal{L}$ and
$\mathcal{L}^*$ is the same as in
\cite{lionsCCDF,cardaliaguet2013mfg}. The minor adaptations needed can
be extracted from the uniqueness proof for nonlocal couplings given in
Appendix \ref{sec:pf_uniq}.

\section{Fractional heat kernel estimates}
\label{sec:fracHK}
Here we introduce fractional heat kernels and prove $L^1$-estimates of their 
spatial derivatives. These estimates are used for the regularity 
results of Sections \ref{sec:fracHJB}, \ref{sec:fracFP}, and \ref{sec:loc}. 
The heat kernel of an elliptic operator $\mathcal A$ is the
fundamental solution of  
$\partial_{t} u - \mathcal{A} u = 0$, or 
$u = \mathcal{F}^{-1} ( e^{t \hat{\mathcal{A}}} )$, 
where $\hat{\mathcal{A}}$ is the Fourier multiplier defined by 
$\mathcal{F} ( \mathcal{A} u ) = \hat{ \mathcal{A}} \hat{u}$.
Taking the Fourier transform of \eqref{Levy_operator}, 
a direct calculation (see \cite{applebaum2009levy}) shows that 

\begin{align*}
    \mathcal{F} \big(\mathcal{L} u \big) = \hat{\mathcal{L}} (\xi) \hat{u} (\xi),
\end{align*}
where 
\begin{align}
    \hat{\mathcal{L}} (\xi) =   \int_{\R^d}  \big( e^{i x\cdot \xi} - 1 - i \xi \cdot z 1_{|z| <1} \big) \,d\mu (z).
    \label{Fourier_multiplicator}
\end{align}
We can split $\hat{\mathcal{L}}$ into a singular and a non-singular part,
\begin{align}
    \hat{\mathcal{L}} (\xi) = \bigg( \int_{|z|<1} + \int_{|z| \geq 1} \bigg) \big( e^{i x\cdot \xi} - 1 - i \xi \cdot z 1_{|z| <1} \big) \,d\mu (z) = \hat{\mathcal{L}}_s (\xi) + \hat{\mathcal{L}}_n (\xi).
    \label{Fourier_symbol}
\end{align}
Note that since $\mu \geq 0$, $\text{Re } \hat{\mathcal{L}} = \int \big( \cos(z\cdot \xi ) -1 \big) \,d\mu \leq 0$.

We will need the heat kernels $K_\sigma$
and $\tilde K_\sigma$ of $\mathcal{L}$ and $\mathcal{L}_s$:
\begin{align}
K_{\sigma} (t,x) = \mathcal{F}^{-1} \big(e^{t \hat{\mathcal{L}}
  (\cdot)} \big)\qquad \text{and} \qquad \tilde{K}_{\sigma} (t,x) = \mathcal{F}^{-1} \big( e^{t \hat{\mathcal{L}}_s (\cdot)} \big).
    \label{heat_kernel_def}
\end{align}
By the L\'evy-Kinchine theorem (Theorem 1.2.14 in
\cite{applebaum2009levy}), $K_{\sigma}$ and $\tilde K_{\sigma}$ are
probability measures for $t>0$:
\begin{align*}
  K_{\sigma},\tilde K_{\sigma}\geq 0\qquad \text{and}\qquad
  \int_{\R^d} K_{\sigma}(x,t)\,dx=1=\int_{\R^d} \tilde K_{\sigma}(x,t)\,dx.
\end{align*}
When \ref{L2'} holds, $\text{Re } \hat{\mathcal{L}}$ and $\text{Re }
\hat{\mathcal{L}}_s \leq -c |\xi|^{\sigma}$ for $|\xi| \geq 1$, and
$K_{\sigma}$ and $\tilde K_{\sigma}$ are absolutely continuous since
$|e^{t\hat{\mathcal{L}}(\cdot)}|$ decays exponentially at infinity. 
An immediate consequence of assumption \ref{L2} is existence for the
 corresponding fractional heat equation. 
\begin{proposition}   \label{heat_eqn_cauchy}
    Assume \ref{L1}, \ref{L2}, $u_0 \in L^{\infty} (\R^d)$, and let $u \left( t,x
    \right) = K_{\sigma} \left( t,\cdot \right) \ast u_0 \left( x
    \right)$. Then $u \in C^{\infty} \big( (0,T) \times \R^d \big)$
    and $u$ is a classical solution of
\begin{align*}
          \partial_t u - \mathcal{L} u = 0\quad \text{in}\quad \R^d\times (0,T), \qquad u \left( 0,x \right)  = u_0 \left( x \right)\quad \text{in}\quad \R^d.
\end{align*}
\end{proposition}

We first show that sums of operators $\mathcal L_i$ satisfying \ref{L1}
and \ref{L2} also satisfy \ref{L1} and \ref{L2}. Let
\begin{align}\label{Lsum}\mathcal L=\mathcal L_1 + \dots +\mathcal L_N
  \quad\text{where} \quad \mathcal L_i u (x) = \int_{Z_i} u (x+z) - u(x) - Du (x) \cdot z
  1_{|z| <1}\ d \mu_i (z),
  \end{align}
$Z_i$ is a $d_i$-dimensional subspace, $\oplus_{i=1}^N Z_i=\R^d$, and 
 $\mathcal L_i$ satisfy \ref{L1} and \ref{L2} in $Z_i$: 
\medskip

\begin{description}
\item[(L1'')\label{L1''}]
  (i) $Z_i\simeq
      \R^{d_i}$ is a subspace for $i=1,\dots,N$, and $\oplus_{i=1}^M Z_i=\R^d$ for $M\leq N$.
\medskip

\quad (ii) $\mu_i\geq0$ is a Radon measure on $Z_i$ satisfying
  $\int_{Z_i} 1 \wedge |z|^2 \,d\mu_i ( z ) < \infty$.\smallskip
  \medskip

\item[(L2'')\label{L2''}] (i) $\mu_i$ satisfy the upper bound \eqref{u-bnd} with $\sigma=\min_i\sigma_i$.
  \medskip

\quad (ii) There are $\sigma_i \in (1,2)$ and $c_i
>0$ such that the heat kernels $K_i$ and $K_i^*$ of

\qquad\ \ $\mathcal L_i$ and $\mathcal L_i^*$ satisfy for $p\in[1,\infty)$, 
  $\beta\in \N_0^{d_i}$, $i=1,\dots,M$, and $t\in(0,T)$,
        \begin{align*}
           \qquad\qquad\qquad \|D_{z_i}^{\beta} K_i
  (t,\cdot) \|_{L^p (Z_i)}+\|D_{z_i}^{\beta} K_i^*
  (t,\cdot) \|_{L^p (Z_i)} \leq c_i t^{-\frac1{\sigma_i}\big({|\beta|}+(1-\frac1p)d\big)}.
        \end{align*}
\end{description}

First observe that here $\mu=\sum_i\mu_i \delta_{0,Z_i^\perp}$ where
$\delta_{0,Z_i^\perp}$ is the delta-measure in $Z_i^\perp$ centered at
$0$. It immediately follows that \ref{L1''} and \ref{L2''} imply
\ref{L1} and \ref{L2} (i). 
\begin{thm}\label{thm_sum}
Assume \ref{L1''}, \ref{L2''} (ii), and $\mathcal L$ is defined in
\eqref{Lsum}. Then the heat kernel $K$ and $K^*$ of $\mathcal L$ and $\mathcal L^*$ belongs to
$C^{\infty}$ and satisfy \ref{L2} (ii) with
$\sigma=\min_i\sigma_i$, i.e. 
$$\|D_x^{\beta} K
  (t,\cdot) \|_{L^p (\R^d)} +\|D_x^{\beta} K^*
  (t,\cdot) \|_{L^p (\R^d)} \leq c_{\beta,T} t^{-\frac1{\sigma_i}\big({|\beta|}+(1-\frac1p)d\big)} \quad
\text{for}\quad t\in(0,T),\quad \beta\in\N_0^d.$$
\end{thm}
\begin{proof}
  First note that in this case $K(t)=\mathcal F^{-1}(e^{t\hat{\mathcal
      L}_1}\cdots e^{t\hat{\mathcal L}_N})=\mathcal K_1(t)*\dots
  *\mathcal K_N(t)$ where
  $$\mathcal K_i(t):=\mathcal F^{-1}_{\R^d}(e^{t\hat{\mathcal
  L}_i})=K_i(t)\delta_{0,Z_i^\perp},\quad K_i(t)=\mathcal
  F^{-1}_{Z_i}(e^{t\hat{\mathcal L}_i}).$$
 For $t\in(0,T)$, \ref{L2''} (ii) implies that
 $$\|D_{z_i}^\beta\mathcal K_i(t)\|_{L^p(\R^d)}=\|D_{z_i}^{\beta} K_{i}
  (t,\cdot) \|_{L^p (Z_i)} \leq c_i t^{-\frac1{\sigma_i}\big({|\beta|}+(1-\frac1p)d\big)}\leq
 c_T t^{-\frac1{\sigma}\big({|\beta|}+(1-\frac1p)d\big)}\ (\sigma\leq\sigma_i)$$
 for some constant $c_T>0$.
 Since $K_i$ is a probability measure by the L\'evy-Kinchine theorem 
 \cite[Thm 1.2.14]{applebaum2009levy},
 $\|\mathcal
 K_j(t)\|_{L^1(\R^d)}=\|K_j(t)\|_{L^1(Z_j)}=1.$ 
 By properties of mollifiers and Young's inequality for
 convolutions it then follows that
 \begin{align*}&\|D_{z_i}^{\beta} K(t,\cdot) \|_{L^p} =\|\mathcal K_1 *
 \dots * D_{z_i}^{\beta} \mathcal K_i* \dots *\mathcal K_N
 \|_{L^p}
 \leq 1\cdot\|D_{z_i}^\beta K_i\|_{L^p(Z_i)}
 \leq c_T t^{-\frac1{\sigma}\big({|\beta|}+(1-\frac1p)d\big)}.
 \end{align*}
 Since $i=1,\dots,M$ was arbitrary and $\oplus_{i=1}^M Z_i=\R^d$, the proof
for $K$ is complete. The proof for $K^*$ is similar.
\end{proof}
 
\noindent It is easy to check that \ref{L2'} implies \ref{L2}(i). We then check that \ref{L2'} implies \ref{L2}(ii).

\begin{thm}    \label{L_heat_kernel_estimate}
  Assume \ref{L1}, \ref{L2'}, and $\mathcal L$ is defined in
  \eqref{Levy_operator}.  Then the 
heat kernels $K$ and $K^*$ of $\mathcal L$ and $\mathcal L^*$ belong to $C^{\infty}$ and 
satisfies \ref{L2}(ii): For $p\in[1,\infty)$, $\beta\in\N_0^d$,
$$\|D_x^{\beta} K
  (t,\cdot) \|_{L^p (\R^d)}+\|D_x^{\beta} K^*
  (t,\cdot) \|_{L^p (\R^d)} \leq c_{\beta,T} t^{-\frac1{\sigma}\big({|\beta|}+(1-\frac1p)d\big)} \quad
\text{for}\quad t\in(0,T).$$
\end{thm}

\begin{example}
  In view of Theorems \ref{thm_sum} and \ref{L_heat_kernel_estimate},
  assumption \ref{L2} is satisfied by e.g. 
  \begin{align*}
    &\mathcal L_1=-(-\Delta_{\R^d})^{\sigma_1/2}-(-\Delta_{\R^d})^{\sigma_2/2},\\
    &\mathcal L_2=-\Big(\!-\frac{\partial^2}{\partial
      x_1^2}\Big)^{\sigma_1/2}-\dots-\Big(\!-\frac{\partial^2}{\partial
      x_d^2}\Big)^{\sigma_d/2},\\
    &\mathcal L_3u(x)= \int_{\R}u(x+z)-u(x)-u'(x)z1_{|z|<1}
    \frac{Ce^{-M z^+ - G z^-}}{|z|^{1+Y}},\\
    &\qquad \text{where $C,G,M>0$,
    $Y\in(0,2)$},\qquad [\text{CGMY model in Finance}]\\
    &\mathcal L_4= \mathcal L + L \quad\text{where $\mathcal L$
      satisfy \ref{L2} and $L$ is any other L\'evy operator.}
  \end{align*}
   We can even take $L$ to be any local L\'evy
      operator (e.g. $\Delta$) if we relax the definition of $\mathcal
      L_i$ to $\mathcal L_i u(x)=\mathrm{tr}[a_i D^2u]+b_i\cdot Du +\int_{Z_i}
      u(x+z)-u(x)-Du(x)\cdot z 1_{|z|<1}\,d\mu_i(s)$ for $a_i\geq0$.
  \end{example}

\begin{remark}(a) \ref{L2} holds also for very non-symmetric
  operators where $\mu$ has support in a cone in $\R^d$. Examples are
  Riesz-Feller operators like 
  $$\mathcal L_3u(x)=
  \int_{z>0}u(x+z)-u(x)-u'(x)z1_{z<1}\frac{dz}{z^{1+\alpha}},\quad \alpha\in(0,2).$$
  We refer to \cite{AK15} for results and discussion, see e.g. Lemma 2.1 (G7) and Proposition 2.3. 
  \smallskip

  \noindent (b) More general conditions implying \ref{L2} can be derived
  from the very general results on derivatives of heat semigroups in
  \cite{SSW12} and heat kernels in \cite{grzywny2017estimates}. Such
  conditions could include more non absolutely continuous and
  non-symmetric L\'evy measures.
  \end{remark}

\noindent We will now prove Theorem \ref{L_heat_kernel_estimate} and start by
proving the result for $\tilde{K}_{\sigma}$, the kernel of $\mathcal{L}_{s}$.
\begin{lemma}     \label{singular_heat_kernel_estimate}
Assume \ref{L1} and \ref{L2'}. Then $\tilde{K}_{\sigma} \in C^{\infty}$, and
for all $\beta\in\N_o^d$ and  $p\in[1,\infty)$, there is $c>0$ such that
$$\| D_x^{\beta} \tilde{K}_{\sigma}(\cdot,t) \|_{L^p (\R^d)} \leq c
t^{-\frac{1}{\sigma}\big(|\beta|+(1-\frac1p)d\big)}\qquad\text{for all}\qquad
t>0.$$
\end{lemma}

\begin{remark}\
(a) When $p=1$, the bound simplifies to $\| D_x^{\beta}
\tilde{K}_{\sigma}(\cdot,t) \|_{L^1 (\R^d)} \leq c
t^{-\frac{|\beta|}{\sigma}}$.

\smallskip
\noindent (b) When $|\beta|=1$, the bound is locally integrable in $t$
when $1\leq p<p_0:=\frac{d}{1+d-\sigma}$. Note that $p_0>1$.
\end{remark}
\begin{proof}
    We verify the conditions of Theorem $5.6$ of
    \cite{grzywny2017estimates}. By \ref{L2'}, assumption ($5.5$) in
    \cite{grzywny2017estimates} holds with 
    \begin{align*}
     \nu_0 ( |x| ) = 
     \begin{cases}
         \frac{1}{|x|^{d+\sigma}}, & |x| < 1, \\ 
         0, & |x| \geq 0.
     \end{cases}
     \end{align*}
   Then we compute the integral $h_0$,
   \begin{align*}
       h_{0} ( r ) : = \int_{\R^d} 1 \wedge \frac{|x|^2}{r^2} \nu_0
       (|x|) \,dx =
       \begin{cases}
           c_{d} ( \frac{1}{2-\sigma} + \frac{1}{\sigma} ) r^{-\sigma} - \frac{c_{d}}{\sigma}, & r < 1, \\[0.2cm]
           c_{d} \frac{1}{2 - \sigma} r^{-2}, & r \geq 1,
       \end{cases}
   \end{align*}
where $c_d$ is the area of the unit sphere. Note that $h_0$ is
positive,
strictly decreasing, and that $ h_0 (r) \leq \lambda^{\sigma} h_{0} (
\lambda r )$ for $0 < \lambda \leq 1 $ and every $r>0$. Hence the scaling
condition (5.6) in \cite{grzywny2017estimates} holds with $C_{h_0}=1$ for
any $\theta_{h_0}>0$. The inverse is given by 
\begin{align*}
    h_0^{-1} (\rho) = 
    \begin{cases}
        \bigg( \frac{(2-\sigma) \rho}{c_{d}} \bigg)^{-\frac{1}{2}}, & \rho \leq \frac{c_{d}}{2-\sigma}, \\
        \bigg( \frac{\rho}{c_{d}} + \frac{1}{\sigma} \bigg)^{-\frac{1}{\sigma}} \bigg( \frac{\sigma ( 2 - \sigma )}{2} \bigg)^{- \frac{1}{\sigma}}, & \rho \geq \frac{c_{d}}{2-\sigma} .
    \end{cases}
\end{align*}
In both cases $t \leq ( 2-\sigma ) / c_{d}$ and $t \geq ( 2-\sigma ) /
c_{d}$, we then find that
$h_0^{-1} (1 / t ) \leq (\tilde{c} t)^{1 / \sigma}$,
where $\tilde{c}$ only depends on $\sigma$ and $d$.

At this point we can use Theorem $5.6$ in \cite{grzywny2017estimates}
to get the following heat kernel bound:
\begin{align*}
    \big| \partial_x^{\beta} p (t,x+tb_{ [h_0^{-1} (1/t) ]} ) \big|
    \leq C_0 [h_0^{-1} (1/t) ]^{-|\beta|} Y_t (x) = C_{0,\sigma} t^{-
      \frac{|\beta|}{\sigma}} Y_t (x),
\end{align*}
for any $t>0$, where $b_r$ does not depend on $x$,
\begin{align*}
    Y_t (x) =  [h_0^{-1} (1/t)]^{-d} \wedge \frac{tK_0 (|x|)}{|x|^d},
\end{align*}
and 
\begin{align*}
    K_{0} ( r ) & := r^{-2} \int_{|x| < r} |x|^{2} \nu_{0} ( |x| ) dx= \frac{c_{d}}{2-\sigma} \cdot
    \left\{\begin{array}{l}
        .\!\! r^{-\sigma} , \ \ r<1 \\[0.2cm]
        \!\! r^{-2}, \ \  r \geq 1
    \end{array}\right\} \leq \frac{c_{d}}{2-\sigma} r^{-\sigma}.
\end{align*}
An integration in $x$ then yields for $p\in[1,\infty)$,
\begin{align}\label{p-Lp}
    \| \partial_x^{\beta} p (t, \cdot) \|_{L^p (\R^d)}^p \leq
    C_{0,\sigma}^p \tilde{c}^p t^{- \frac{p|\beta|}{\sigma}} \int_{R^d} Y_t (x)^p \,dx.
\end{align}

We compute $\|Y_t\|_{L^p (\R^d)}$. Since $h_0^{-1} (1/t)   \leq \tilde{c} t^{1/\sigma}$ and $K_0 (r) \leq \frac{c_{d}}{2- \sigma} r^{- \sigma} $, we can compute the minimum to find a constant $c_{\sigma, d} >0$ such that 
\begin{align*}
   0\leq Y_t (x) \leq \begin{cases}
        (\tilde ct)^{-d/\sigma} , & \text{ for $|x| < c_{\sigma, d} t^{1/\sigma}$} \\
        \frac{c_d}{2-\sigma} \frac{t}{|x|^{d+\sigma}}, & \text{ otherwise.}
    \end{cases}
\end{align*}
A direct computation then shows that 
$$\int_{\R^d} Y_t (x)^p \,dx \leq  c_{d, \sigma,p} t^{-\frac{(p-1)d}{\sigma}},$$
where $c_{d, \sigma,p}>0$ is a constant not depending on
$t$. Combining this estimate with \eqref{p-Lp} concludes the proof of
the Lemma. 
\end{proof}

\begin{proof}[Proof of Theorem \ref{L_heat_kernel_estimate}]
  result for $K_\sigma$ follows by Lemma
  \ref{singular_heat_kernel_estimate} and a simple computation:
    \begin{align*}
        \big\|D_x^{\beta} K_{\sigma} \big\|_{L^p} &= \big\|D_x^{\beta}
        \mathcal{F}^{-1} \big( e^{t \hat{\mathcal{L}}_s}e^{t
          \hat{\mathcal{L}}_n} \big)\big\|_{L^p}= \big\| \Big(
        D_x^{\beta} \mathcal{F}^{-1} \big( e^{t \hat{\mathcal{L}}_s}
        \big) \Big) \ast \mathcal{F}^{-1} \big( e^{t
          \hat{\mathcal{L}}_n} \big) \big\|_{L^p} \\
        &\leq \|D_x^{\beta} \mathcal{F}^{-1} \big( e^{t
          \hat{\mathcal{L}}_s} \big) \|_{L^p} \int_{\R^d}
        \mathcal{F}^{-1} \big( e^{t \hat{\mathcal{L}}_n} \big) \leq
        ct^{-\frac{1}{\sigma}\big(|\beta|+(1-\frac1p)d\big)}\cdot 1.
    \end{align*}
The last integral is 1 since $\mathcal{F}^{-1} \big( e^{t
  \hat{\mathcal{L}}_n} \big)$ is a probability by
e.g.~Theorem 1.2.14 in \cite{applebaum2009levy}. Since
$\mathcal L^*$ is an operator of the same type as $\mathcal L$ with a
L\'evy measure $\mu^*$ also satisfying \ref{L1} and \ref{L2'} (cf. Lemma
\ref{adjoint_operator}), the
computations above show that $K^*_\sigma$ also
satisfy the same bound as $K_\sigma$.
\end{proof}
\begin{remark}\label{rem:smoothing}
  From this proof it follows that the smoothing properties of $\LL$ and
  $K_\sigma$ are independent of $\hat\LL_n$ and then also $\mu 1_{|z|>1}$.
  \end{remark}

By interpolation we obtain estimates for fractional derivatives of the heat kernel.
\begin{proposition}    \label{kernel_estimates}
  Assume \ref{L1}, \ref{L2}, $t\in[0,T]$, $s,\sigma \in \left( 0,2 \right)$, and $|D|^s:= \left(
    -\Delta \right)^{s/2}$. Then
    \begin{align*}
        \| |D|^{s} K_{\sigma} \left( t \right) \|_{L^1 \left( \R^d \right)} &\leq c t^{-\frac{s}{\sigma}},
    \end{align*}
    and if $s \in \left( 0,1 \right)$, then
    \begin{align*}
        \| |D|^s  \partial_x K_{\sigma} \left( t \right) \|_{L^1 \left( \R^d \right)} &\leq c t^{-\frac{s+1}{\sigma}}. 
    \end{align*}
\end{proposition}
\begin{proof}
    By Remark
    \ref{fractional_laplace_estimate} (a) with $p=1$ and \ref{L2}, we find that
    \begin{align*}
        \int ||D|^s K^{\sigma} (t)| \,dx &\leq c \|D^2 K^{\sigma} (t) \|_{L^1}^{\frac{s}{2}} \|K^{\sigma} \|_{L^1}^{1-\frac{s}{2}} \leq \big( ct^{-\frac{2}{\sigma}} \big)^{s/2}  1^{1-s/2} \leq ct^{-\frac{s}{\sigma}}.
    \end{align*} 
The proof of the second part follows in a similar way from Remark
    \ref{fractional_laplace_estimate} (b).
\end{proof}

\section{Fractional Hamilton-Jacobi-Bellman equations}
\label{sec:fracHJB}
Here we prove regularity and well-posedness for solutions of the fractional
Hamilton-Jacobi equation. In our proof we use heat kernel estimates
(Section \ref{sec:fracHK}), a Duhamel formula, and a fixed point
argument as in \cite{imbert2005non, droniou2003global}. The fractional
Hamilton-Jacobi equation is given by 
\begin{align}
   \left\{
\begin{array}{ll}
    \partial_t u - \mathcal{L} u + H \left( x,u, Du \right) = f \left( t,x \right) & \text{ in } \   \left(0,T \right) \times \R^d, \\
  u \left( 0,x \right) = u_0 \left( x \right) & \text{ in } \R^d,
\end{array}  
\right. 
\label{fractal_HJ}
\end{align}
where $f$ is the source term and $u_{0}$ initial condition. We assume
\smallskip

\begin{description}
    \item[(B1)\label{B1}]  $u_{0} \in C_{b} ( \R^d  )$ and $f\in C_{b} ([0,T]\times \R^d  )$.
      \medskip
      
    \item[(B2)\label{B2}]  There is an $L>0$ such that for
      all $x,y \in \R^d$, $t \in  [ 0,T]$,
        $$| f ( t,x ) -
      f ( t,y ) | + |u_0(x)-u_0(y)| \leq L|x-y|.$$
\end{description}
\smallskip

Assumptions \ref{L1}, \ref{A3}--\ref{A5}, \ref{B1}--\ref{B2} implies
that there exists a bounded 
$x$-Lipschitz continuous viscosity solution $u$ of (\ref{fractal_HJ})
(cf. e.g \cite{jakobsen2005continuous, jakobsen2006maximum,BaIm2008,imbert2005non}).

\begin{thm}[Comparison principle]  \label{comp_visc_sol}
    Assume \ref{L1}, \ref{A3}--\ref{A5}, \ref{B1}--\ref{B2} and
   $u,v$ are viscosity sub- and supersolutions of
  (\ref{fractal_HJ}) with bounded continuous initial data
  $u_0,v_0$.  If $u_0\leq v_0$ in $\R^d$, then $u\leq v$ in $\R^d\times (0,T)$. 
\end{thm}
\begin{proof}[Outline of proof]
If $u$ and $v$ are uniformly continuous, then the proof is essentially
the same as the proof of Theorem 2 in \cite{imbert2005non}.
When $u$ and $v$ are not uniformly continuous, the limit (13) in
\cite{imbert2005non} no longer holds because (in the notation of \cite{imbert2005non}) $\frac{|\bar x-\bar
  y|^2}{\varepsilon}\not\to 0$. However, this can
be fixed under our assumptions, loosely speaking because we 
can remove all $O(\frac{|\bar x-\bar
  y|^2}{\varepsilon})$-terms before taking limits by modifying the
test function. The modification consists in introducing an exponential
factor in the quadratic term: $\frac{e^{Ct}}{\varepsilon}|x-y|^2$ for
$C$ large enough. 
\end{proof}

\begin{remark}
We drop a complete proof here for two reasons: (i) it is long and
rather standard, and (ii) we only apply the result in cases where $u$
and $v$ are uniformly continuous and an argument like in
\cite{imbert2005non} is sufficient.
  \end{remark}

  \begin{thm}[Well-posedness] 
  Assume \ref{L1}, \ref{A3}--\ref{A5}, and \ref{B1}--\ref{B2}.
  \smallskip

  \noindent (a) There exists a (unique)
   bounded continuous viscosity solution $u$ of
  (\ref{fractal_HJ}) in $\left( 0,T \right) \times \R^d$ 
  such that $u \left( 0,x \right) = u_0 \left( x \right)$.
  \medskip
  
  \noindent (b) $\|u\|_\infty \leq \|u_0\|_\infty + C_0T$ where
  $C_0 := \| H ( \cdot, 0,0 )  \|_\infty + \| f  \|_\infty  $ is finite by \ref{A3} and \ref{B1}.
  \medskip

  \noindent (c) If also $u_0 \in W^{1,\infty} \left( \R^d \right)$, then
$$\|D u \left( t,\cdot \right)\|_{L^{\infty} \left( \R^d \right)} \leq
M_T,$$   
where $M_T = \mathrm{e}^{2C_R T} \left( \frac12C_R + T^2\| D_x f
  \|_{\infty}^2 + \| D u_0
  \|_{\infty}^2 \right)^{1/2}$, with $C_R$ from \ref{A4} and $R = \|u\|_\infty$. 
\label{existence_visc_sol}
\end{thm}
\begin{proof}
The proof of (a) is quite standard and almost identical to the proof of
  Theorem 3 in \cite{imbert2005non}. Part (b) follows from comparison,
  Theorem \ref{comp_visc_sol}, and the proof of part (c) is similar to
  the proof of Lemma 2 in \cite{imbert2005non}.  
\end{proof}
Using parabolic regularity (in the form of \ref{L3}
\cite{barles2012lipschitz}) and the method of Ishii-Lions, it is 
possible to obtain Lipschitz bounds that only depend on the
$C_b$-norm of $f$:
 \begin{thm}\label{IshiiLions}
  Assume \ref{L1}, \ref{L3}, \ref{A3}--\ref{A5}, $f\in C_b([0,T]\times
  \R^d)$ and $u_0\in W^{1,\infty}(\R^d)$. Then the viscosity solution
  $u$  of  \eqref{fractal_HJ} is Lipschitz continuous in $x$ and there
  is a constant $M>0$ such that for all $t\in[0,T]$,
$$\| Du \left( t,\cdot \right)\|_{L^{\infty} \left( \R^d \right)} \leq
M,$$   
where $M$ depends on $\|u\|_\infty$, $\|f\|_\infty$, $d$, and the
quantities in \ref{A3}--\ref{A5}.
 \end{thm}
 \begin{proof}
   In the periodic case this result is a direct consequence of
   Corollary 7 in \cite{barles2012lipschitz}. The original proof is
   for a right-hand side $f$ not depending on $t$. For contiuous $f =
   f ( x,t )$, the proof is exactly the same. The result also holds in
   the whole space case, and this is explained in section 5.1 (see
   Theorem 6 for the stationary case).
   \end{proof}
 Similar parabolic results for the whole
   space are also given in \cite{CLN19}. 
To have classical solutions we make further regularity assumptions on the data:

\smallskip
\begin{description}
    \item[(B3)\label{B3}] $\| f ( t,\cdot )  \|_{C_{b}^{2} ( \R^d  )}
      \leq C $ for all $t \in [ 0, T]$.

        \smallskip
    \item[(B4)\label{B4}]  $u_0 \in C_{b}^{3}( \R^d )$.
    \medskip
\end{description}
Note that $f$ needs less spatial regularity than $H$ in \ref{A3}.
\begin{thm}[Classical solutions]  \label{thm_HJ}
    Assume \ref{L1}--\ref{L2}, \ref{A3}--\ref{A5}, and \ref{B1}--\ref{B4}. Then
  \eqref{fractal_HJ} has a unique classical solution $u$ such that
      $\partial_t u , u, Du, \cdots, D^{3} u \in C_b \big( ( 0,T )
  \times \R^d \big)$ with
  \begin{align*}
      \|\partial_{t} u \|_{L^{\infty}} + \|u \|_{L^{\infty}} + \|D u \|_{L^{\infty}} + \ldots + \|D^{3} u \|_{L^{\infty}} \leq c,
  \end{align*}
  where $c$ is a constant depending only on $\sigma$, $T$, $d$, and
  quantities from \ref{L1}--\ref{B4}.
\end{thm}

To have space-time uniform continuity (and compactness) of
derivatives, we assume: 
\smallskip

\begin{description}
    \item[(B5)\label{B5}] There is a modulus of continuity
      $\omega_{f}$ such that for all $x,y\in\R^d$, $t,s\in[0,T]$,
      $$| f ( s,x ) - f ( t,y ) | \leq \omega_{f} ( | s-t | + | x-y |).$$
\end{description} 
\smallskip

\begin{thm}[Uniform continuity]  \label{HJ_Holder_thm}
  Assume \ref{L1}--\ref{L2}, \ref{A3}--\ref{A5}, and \ref{B1}--\ref{B5}. Then the unique classical solution $u$ of \eqref{fractal_HJ} also satisfies 
\begin{align}
    &|u ( t,x ) - u ( s,y ) | + |Du ( t,x ) - Du ( s,y ) | +  |D^2u ( t,x ) - D^2u ( s,y ) | \nonumber \\ 
    &+ |\partial_t u(t,x)-\partial_t u(s,y)|+ |\mathcal L u(t,x)- \mathcal L u(s,y)| \leq \omega(|t-s|+|x-y|), 
\end{align}
where $\omega$ only depends on 
$\sigma$, $T$, $d$, and quantities from
\ref{L1}--\ref{B5}.
\end{thm}

\begin{remark} 
    Imbert shows in \cite{imbert2005non} that when $\mathcal{L} = -
    (-\Delta)^{\sigma/2}$, $f \equiv 0$, and $u_0 \in W^{1,\infty} (
    \R^d )$, there exists a classical solution $u$ such that
      $ \|u\|_{C_b} + \|Du\|_{C_b} + \|t^{1/\sigma}
    D^2 u\|_{C_b} \leq c$. He goes on to show that when $H=H(p)\in C^\infty$, then
    $u\in C^\infty$.
 In this paper we prove results for a much
    larger class of equations and nonlocal operators. Our results are
    also more precise: We need and prove time-space uniform continuity
    of all derivatives appearing in the equation, see Theorem
    \ref{HJ_Holder_thm}. To do we need a finer analysis of the
    regularity in time. A final difference is 
    that  our estimates do not blow up as $t\to 0^+$.
    Note that it is easy to adapt our proofs and obtain even higher order
    regularity, e.g. treat the case $H = H(x,u,p) \in C^{\infty}$.
\end{remark}

 To prove Theorem \ref{thm_HJ} and \ref{HJ_Holder_thm}, we first
restrict ourselves to a short time interval.

\subsection{Short time regularity by a Duhamel formula}
Let $K$ be the fractional heat kernel defined in
\eqref{heat_kernel_def}. 
A solution $v$ of (\ref{fractal_HJ}) is
formally given by the Duhamel formula
\begin{align}
    \begin{split}
&v(t,x)=    \psi \left( v \right) \left( t,x \right) \\ 
    &:= K \left( t, \cdot \right) \ast v_0 \left( \cdot \right) \left( x \right) - \int_0^t K \left( t-s, \cdot \right) \ast \left( H \left(s,\cdot, v \left( s,\cdot \right), Dv \left( s, \cdot \right) \right) - f \left( s, \cdot \right) \right) \left( x \right) ds,
\end{split}
    \label{duhamel_integral}
\end{align}
where $\ast$ is convolution in $\R^d$.
Note that solutions of this equation are fixed points of $\psi$. 
\begin{proposition}[Spatial regularity]     \label{duhamel_regularity}
        Assume \ref{L1}--\ref{L2}, \ref{A3}--\ref{A5},
        \ref{B1}--\ref{B3}, and $k \in \{2,3\}$. For $R_0\geq0$, let
        $R_1 =  \left( 1+\mathcal{K} \right) R_0 + 1$ with
        $\mathcal{K}$ defined in \ref{L2}. \medskip

\noindent (a) If $v_0 \in W^{k-1,\infty} ( \R^d)$ with
        $\|v_0\|_{W^{k-1,\infty}} \leq R_0$, then there is $T_{0} \in ( 0, T)$
         such that $\psi$ in
        (\ref{duhamel_integral}) has a unique  
        fixed point $v \in C_b^{k-1} ( [ 0,T_0 ] \times \R^d)$ with $t^{1/\sigma}D^{k} v \in C_b ([ 0,T_0 ]\times \R^d)$ and
        $$\|v\|_{L^{\infty}} + \dots + \|D^{k-1} v\|_{L^{\infty}} +
        \|t^{1/\sigma} D^{k} v\|_{L^{\infty}} \leq R_1.$$

        \noindent (b) If $v_0 \in W^{k,\infty} ( \R^d)$ with
        $\|v_0\|_{W^{k,\infty}} \leq R_0$, then there is $T_{0} \in ( 0, T)$
         such that $\psi$ in
        (\ref{duhamel_integral}) has a unique  
        fixed point $v \in C_b^{k} ( [ 0,T_0 ] \times \R^d)$  and
        $$\|v\|_{L^{\infty}} + \dots + \|D^{k} v\|_{L^{\infty}}\leq R_1.$$

\noindent In both cases  $T_0$ only depends on $\sigma$ and the quantities in 
        \ref{L1}--\ref{B3}.
\end{proposition}
\begin{proof}
(a) We will use the Banach fixed point theorem in the Banach (sub) space
\begin{align*}
    X = \big\{ v : v,\dots, D^{k-1} v, t^{1/\sigma} D^{k} v \in C_b ( [ 0,T_0 ] \times \R^d ) \text{ and } \vertiii{v}_{k} \leq R_1 \big\},
\end{align*}
where $\vertiii{v}_k =\|v\|_{k-1}  + \sum_{|\beta| = k} \|t^{1/\sigma} D_x^{\beta}v
\|_{\infty}$ and $\|v\|_k = \sum_{0 \leq |\beta| \leq k}
\|D_x^{\beta}v \|_{\infty}$.

Let $v \in X$. For $i = 1,\ldots,d$ and $\beta \in \N^d$,  $|\beta|
\leq k-2$,
\begin{align}\label{likn1}
        \partial_x^{\beta} \partial_{x_i} \psi (v) = K ( t ) \ast \partial_x^{\beta} \partial_{x_i} v_0(x) - \int_0^t \partial_{x_i} K \big( t-s \big) \ast \partial_x^{\beta} \Big(H \big(\cdot,\cdot, v  , Dv \big) - f \Big) (s,x)\, ds,
    \end{align}
while for $|\beta| = k-1$,
\begin{align}\label{likn2}
         t^{1/\sigma} \partial_x^{\beta} \partial_{x_i} \psi \left( v
         \right) = &\ t^{1/\sigma} \partial_{x_i} K \left( t \right) \ast \partial_x^{\beta} v_0 \left( x \right) \\
      \notag   & - t^{1/\sigma} \int_0^t \partial_{x_i} K \left( t-s \right) \ast \partial_x^{\beta} \Big( H \big(\cdot,\cdot, v , Dv \big) - f  \Big) \left(s, x \right) ds.
     \end{align}
If $w$ and $F$ are bounded functions, then 
$K \left( t,\cdot \right) \ast w$ and $\int_0^t
\partial_{x} K \left( t-s, \cdot \right) \ast F \left( s,\cdot \right) ds$
are well-defined, bounded and continuous by \ref{L2} and 
an argument like in the proof of \cite[Proposition 3.1]{droniou2003global}.
It follows by \ref{A3} and \ref{B3}, that the derivatives of $\psi
\left( v \right)$ in \eqref{likn1} and \eqref{likn2}
are well-defined, bounded, and continuous. In particular by \ref{L2},
for $t \in ( 0, T)$, 
$$\|t^{1/\sigma} \partial_{x_i} K \left( t \right) \ast
\partial_x^{\beta} v_0\|_{C_b}\leq \mathcal K\|\partial_x^{\beta} v_0\|_{C_b}.$$

Let $u,v \in X$ and $t\in[0,T_0]$. By \ref{A3} and \ref{B3}
there is a constant $C_{R_1} >0$, such that
\begin{align} \label{H_bounds}
        &\big|\partial_x^{\beta} \big[H \left(s,x,u ( s,x ), Du
        ( s,x ) \right)\big]\big| + \big| \partial_x^{\beta} f (
        s,x ) \big|
        \leq  \begin{cases}
            C_{R_1} , \qquad\qquad\, 0 \leq |\beta| \leq k-2, \\
            C_{R_1} \left( 1+ s^{-\frac1\sigma} \right) , \quad|\beta| = k-1,
            \end{cases} \\[0.2cm]
        \label{H_difference}
              &\big|\partial_x^{\beta} \big[H \left( s,x,u, Du \right)\big] -
        \partial_x^{\beta} \big[H \left( s,x,v, Dv \right)\big] \big|
        \leq  \begin{cases}
             C_{R_1}  \|u-v\|_{|\beta|+1} , \quad\ \ \,0 \leq |\beta| \leq k-2, \\
             C_{R_1} \left( 1+ s^{-\frac1\sigma} \right)
             \vertiii{u-v}_{3} ,  |\beta| = k-1.
            \end{cases}\hspace{-0.55cm}
        \end{align}
    By \ref{L2}
    $\int_0^t \int_{\R^d} |K \left( t-s,x \right) | \,dx\, ds\leq T_0$, 
      $ \int_0^t \int_{\mathbb{R}^d}
    |\partial_{x_i} K \left( t-s,x \right) | \,dx\, ds \leq k \left(
    \sigma \right) T_0^{1- 1/\sigma}$, and
    \begin{align*}
      \int_0^t s^{-1/\sigma}\int_{\mathbb{R}^d}
    |\partial_{x_i} K \left( t-s,x \right) | \,dx \,ds\leq \gamma \left(
    \sigma \right) T_0^{1- 1/\sigma},
    \end{align*}
 where $k \left( \sigma \right) = \mathcal{K} \frac{\sigma}{\sigma-1}$
 and $\gamma(\sigma)=\mathcal{K}\int_0^1 \left( 1-\tau \right)^{-1/\sigma} \tau^{-1/\sigma} d\tau$.
 From these considerations and Young's inequality for convolutions on
 \eqref{likn1} and \eqref{likn2}, we compute the norm  in $X$,
    \begin{align*}
        &\|\psi \left( v \right) \|_{\infty} + \sum_{i=1}^d \Big( \|\partial_i \psi \left( v \right) \|_{\infty} + \sum_{1\leq|\beta| =k-2} \|\partial_{x}^{\beta} \partial_i \psi \left( v \right) \|_{\infty} + \sum_{|\beta| = k-1} \|t^{1/\sigma} \partial_{x}^{\beta} \partial_i \psi \left( v \right) \|_{\infty} \Big) \\
        &\leq \left( 1+\mathcal{K} \right) R_0 \\
        &+\underbrace{C_{R_1} \Big( T_0 + \sum_{i=1}^d \Big( k \left(
        \sigma \right) T_0^{1-1/\sigma} + \sum_{1\leq|\beta| =k-2} k
        \left( \sigma \right) T_0^{1-1/\sigma} + \sum_{|\beta| = k-1}
        k \left( \sigma \right) T_0 +  \gamma(\sigma) T_0^{1-1/\sigma}
        \Big) \Big)}_{=: c (T_0)}. 
    \end{align*} 
    Taking $T_0\in(0,T)$ such that $c (T_0) \leq 1/2$,  $\psi$ maps
    $X$ into itself: By the definition of $R_1$,
    \begin{align*}
        \vertiii{\psi (v)}_{k} \leq (1+\mathcal{K}) R_0 + \frac{1}{2} \leq R_1.
    \end{align*}
  It is also a contraction on $X$.  By (\ref{H_difference}) and
  $\|u\|_1 \leq \|u\|_{k-1} \leq  \vertiii{u}_{k}$,
    \begin{align*}
        &\vertiii{ \psi  \left( u \right) - \psi \left( v \right) }_{k}  \\
        & \ \ \leq C_{R_1} \Big( T_0 \|u-v\|_{1} + \sum_{i=1}^d \Big(
      k \left( \sigma \right) T_0^{1-1/\sigma}\|u-v\|_{1} +
      \sum_{1\leq|\beta| \leq k-2} k \left( \sigma \right) T_0^{1-1/\sigma}\|u-v\|_{|\beta|+1}\\ 
        & \ \ \ \ \ \  + \sum_{|\beta| = k-1} \big( k \left( \sigma
      \right) T_0 +  \gamma(\sigma) T_0^{1-1/\sigma} \big)
      \vertiii{u-v}_{|\beta|+1} \Big) \Big)  \\  
      & \ \ \leq c (T_0) \vertiii{u-v}_{k} \leq \frac{1}{2} \vertiii{u-v}_{k}.
    \end{align*}
An application of Banach's fixed point theorem in $X$ now concludes
the proof of part (a).
\medskip

\noindent (b) We define the Banach (sub) space
\begin{align*}
    X = \big\{ v : v,Dv, \ldots, D^{k} v \in C_b ( ( 0,T_0 ) \times \R^d ) \text{ and } \| v  \|_{k}  \leq R_1 \big\},
\end{align*}
where $\|v\|_k = \sum_{0 \leq |\beta| \leq k}
\|D_x^{\beta}v \|_\infty$. 
We use \eqref{likn1} with $| \beta |\leq k-1 $,
and only the first parts of \eqref{H_bounds} and
\eqref{H_difference}. The rest of the proof is then
similar to the proof of part (a).
\end{proof}

\noindent We proceed to prove time and mixed time-space regularity
results. As a consequence, the solution of \eqref{duhamel_integral} is
a classical solution of \eqref{fractal_HJ}. 
\begin{proposition}     \label{time_proposition}
    Assume $T_0>0$, \ref{L1}--\ref{L2}, \ref{A3}--\ref{A5}, \ref{B1}--\ref{B3}, $v$ satisfies \eqref{duhamel_integral}, and $v,Dv,D^2 v\in C_b ([0,T_0] \times \R^d)$. Then
    \begin{itemize}
        \medskip
    \item[(a)] $\partial_{t} v \in C_b ([0,T_0] \times \R^d)$ and $\| \partial_{t} v  \|_\infty \leq c$, 
        where $c$ depends only on $\sigma,T_{0},d$, the quantities in
        assumptions \ref{L1}--\ref{B3}, and $\|D^kv\|_\infty$ for
        $k=0,1,2$.  
    \end{itemize}
        Assume in addition $D^{3} v\in C_b ([0,T_0] \times \R^d)$.
    \begin{itemize}
      \item[(b)] $v, Dv,\mathcal{L} v, D^{2} v \in UC ( [ 0,T_{0}  ]
        \times \mathbb{R}^{d} )$ with modulus $\omega (t-s,x-y) = C ( |t-s|^{\frac{1}{2}} +  |x-y|)$,
        where $C>0$ only depends on $\sigma,T_{0},d$, the quantities in
        assumptions \ref{L1}--\ref{B3}, and $\|D^kv\|_\infty$ for $k=0,1,2,3$. 
        \medskip
    \item[(c)] If also \ref{B5}, then $\partial_{t} v  \in UC ( (
        0,T_{0} ] \times \mathbb{R}^{d} )$, where the modulus only
        depends on $T_{0}$, $\sigma,T_{0},d$, the quantities in
        assumptions \ref{L1}--\ref{B5}, and the moduli of $v,Dv,
        \mathcal{L} v, D^{2} v$.  
    \end{itemize}
\end{proposition}
\begin{corollary}
  \label{PDE_soln}
    Under the assumptions of Proposition \ref{time_proposition} (a), $v$ is a classical solution of \eqref{fractal_HJ} on $(0,T_0) \times \R^d$.
\end{corollary}
Follows by differentiating formula \eqref{duhamel_integral}. To prove Theorem \ref{time_proposition} we use the Duhamel formula
\begin{align}
    v \left( t,x \right) = K \left( t,\cdot \right) \ast v_0 \left( \cdot \right) \left( x \right) - \int_0^t K \left( t-s, \cdot \right) \ast g \left( s, \cdot \right) \left( x \right) ds,
    \label{time_duhamel}
\end{align}
corresponding to the equation
\begin{align}
    \partial_t v \left( t,x \right) - \mathcal{L} v \left( t,x \right) + g \left( t,x \right) = 0.
    \label{time_HJ}
\end{align}
The following technical lemma is proved in Appendix \ref{pf-DH-lem}.

\begin{lemma}\label{DH-lem} 
    Assume \ref{L1}--\ref{L2}, $g, \nabla g \in C_{b} \big( ( 0,T ) \times
    \mathbb{R}^{d} \big)$, and let
    \begin{align*}
    \Phi ( g ) ( t,x ) = \int_{0}^t K ( t-s, \cdot ) \ast g ( s,
      \cdot ) ( x ) ds. 
    \end{align*}

    \noindent (a) \ $\Phi ( g ) ( t,x )$ is $C^1$ w.r.t. $t \in (
0,T )$ and $\ \partial_{t} \Phi ( g ) ( t,x ) = g ( t,x ) +
\mathcal{L} [ \Phi ( g ) ] ( t,x )$.
\medskip

\noindent (b) \ If $\beta\in(\sigma-1,1)$ and $g \in UC( ( 0,T ) \times \mathbb{R}^{d})$, then
\begin{align*}
    &|\partial_{t} \Phi ( g ) (  t,x ) - \partial_{t} \Phi ( g ) (  s,y
  ) | +|\mathcal L \Phi ( g ) (  t,x ) - \mathcal L \Phi ( g ) (  s,y
    ) |\\
    &\leq   2(1+c)\| g  \|_{C_{b,t}C_{b,x}^{1}}   |x - y
  |^{1-\beta}\\
  &\quad+2(1+c) \| g  \|_{C_{b,t}C_{b,x}^{1}}^{\beta}\omega_{g}( |t-s| )^{1-\beta} +\tilde c\|g\|_{C_{b}}|t-s|^{\frac{\sigma-1}\sigma}, 
\end{align*}
where $c=\frac{\sigma}{\sigma-1}T^{\frac{\sigma-1}\sigma}\mathcal
K\int_{|z|<1}|z|^{1+\beta}d\mu(z)+4T\int_{|z| \geq 1} d \mu ( z )$,
$$\tilde c=2\frac{\sigma}{\sigma-1}\mathcal K\int_{|z|<1}|z|^{1+\beta}d\mu(z)K+2T^{\frac1\sigma}\int_{|z| \geq 1} d \mu ( z
  ),$$ 
and $K=\max_{s,t\in[0,T]}{\big|t^{\frac{\sigma-1}\sigma}-s^{\frac{\sigma-1}\sigma}\big|}/{|t-s|^{\frac{\sigma-1}\sigma}}$.
\end{lemma}
Note that $c$, $\tilde c$, and $K$ are finite. We have the following
results for \eqref{time_duhamel} and \eqref{time_HJ}. 
\begin{lemma}
    \label{time_modulus_lemma}
    Assume \ref{L1}--\ref{L2}, $v$ satisfies \eqref{time_duhamel}, and 
    $v, \nabla v, D^2 v, g, \nabla g \in C_b \left( [ 0,T ] \times \R^d
    \right)$.

    \noindent (a) $\partial_t v \in C_b \left(( 0,T )
    \times \R^d \right)$, and $v$ solves equation \eqref{time_HJ} pointwise.
    \label{lemma_time_regularity}
    \smallskip

    \noindent (b) If in addition $g\in UC ( [ 0,T ] \times \mathbb{R}^{d})$, then $\partial_t v$ and
$\mathcal Lv$ are uniformly continuous and for any $x,y\in\R^d$,
$t,s\in[0,T]$, $k=0,1,2$,
\begin{align}
    |\partial_t v(t,x)-\partial_t v(s,y)|+ |\mathcal L v(t,x)- \mathcal L v(s,y)| \leq \omega(|t-s|+|x-y|),
\end{align}
where $\omega$ only depends on $\omega_g$, $\| g  \|_\infty $, $\| g
\|_{C_{b,t}C_{b,x}^{1}}$, $\| Dv_{0}  \|_{\infty} $ , $\|D^2v_0\|_\infty$, $\sigma$, $T$, and $\mu$.
\end{lemma}

\begin{proof}
(a) \ By the assumptions and Proposition \ref{heat_eqn_cauchy} and Lemma
\ref{DH-lem} (a), we can differentiate the right hand side of 
\eqref{time_duhamel}. Differentiating and using the two results then
leads to
\begin{align*}
    \partial_t v &= \partial_t \left( K \left( t \right) \ast v_0 \right) - \partial_t \int_0^t K \left( t-s \right) \ast g \left( s \right) ds \\
    &= \mathcal{L} \left( K \left( t \right) \ast v_0 \right) - g \left( t \right) - \mathcal{L} \int_0^t K \left( t-s \right) \ast g \left( s \right) ds \\
    &=  - g \left( t \right) + \mathcal{L} \left( K \left( t \right) \ast v_0 - \int_0^t K \left( t-s \right) \ast g \left( s \right) ds \right)  \\
    &= - g \left( t \right) + \mathcal{L} v \left( t \right).
\end{align*}
Thus we end up with (\ref{time_HJ}) and the proof of (a) is complete.
\smallskip

\noindent (b) \ By \eqref{time_duhamel}, $v$ is the sum of two
convolution integrals. The regularity of the second integral follows
from Lemma \ref{DH-lem} (b). The regularity of the
first integral follows by similar but much simpler arguments, this time
with no derivatives on the kernel $K$ (and hence two derivatives on
$v_0$). We omit the details. 
\end{proof}

\begin{proof}[Proof of Proposition \ref{time_proposition}]
    \noindent (a)  In view of the assumptions,
    the result follows from Lemma \ref{lemma_time_regularity}(a) with $g \left( t,x \right) = H \big( x, v ( t,x ) , Dv ( t,x ) \big)
- f(t,x)$. 
\smallskip

\noindent (b) By (a) and Corollary \ref{PDE_soln}, $v$
solve \eqref{fractal_HJ}. 
We show that $D^{2}v \in UC ( [ 0,T ] \times \R^d  )$. 
Let $w = \partial_{x_{i}x_{j}}^{2}  v$ and $w^{\epsilon} =
w \ast \rho_{\epsilon}$ for a standard mollifier
$\rho_{\epsilon}$. 
Convolving \eqref{fractal_HJ} with $\rho_{\epsilon}$ and then
differentiating twice ($\partial_{x_{i}} \partial_{x_{j}}$), we find
that 
\begin{align*}
    \partial_{t} w^{\epsilon} - \mathcal{L} w^{\epsilon} + \partial_{x_{i}x_{j}}^{2} \big(   H ( t,x,v,Dv ) \ast \rho_{\epsilon} \big) = \partial_{x_{i} x_{j}} f \ast \rho_{\epsilon}.
\end{align*}
By Lemma \ref{L_p_bounds} $\|\mathcal{L}
w^{\epsilon}\|_\infty\leq c\| w^{\epsilon}\|_{C^2_b}$, and then by properties of convolutions,
$$\|\mathcal{L} w^{\epsilon}\|_{\infty}\leq
c \sum_{k=2}^4\|D^kv^{\epsilon}\|_{\infty}\leq \frac c\epsilon \|D\rho\|_{L^1}\|D^3v\|_{\infty}+c(\|D^3v\|_{\infty}+\|D^2v\|_{\infty}).$$
It follows that
$    | \partial_{t} w^{\epsilon} | \leq \frac{\tilde c}{\epsilon} + K,$
where $\tilde c=c\|D\rho\|_{L^1} \| D^{3} v  \|_{\infty}$ and $K>0$ is a constant only depending on $\| v  \|_{\infty}, \| Dv  \|_{\infty }, \| D^{2} v  \|_{\infty} , \| D^{3} v  
\|_{\infty}, \| D^2 f  \|_{\infty} $ and $C_{R} > 0$ from \ref{A3}, with 
$R = \max ( \| v  \|_{\infty} ,  \| Dv  \|_{\infty} )$. 
We find that
\begin{align*}
   & \| w ( t ) - w ( s )  \|_{\infty} \leq  \| w^{\epsilon} ( t ) - w ( t ) \|_{\infty} + \| w^{\epsilon} ( t ) - w^{\epsilon} ( s )  \|_{\infty} + 
    \| w^{\epsilon} ( s ) - w ( s )  \|_{\infty} \\
                                      &\leq 2 \| Dw  \|_{\infty} \cdot \epsilon + \| \partial_{t} w^{\epsilon}  \|_{\infty} | t-s |
    \leq 2 \| D^3v  \|_{\infty} \cdot \epsilon + ( \frac{\tilde c}{\epsilon} + K  )  | t-s | \leq  C |t-s|^{\frac{1}{2}}+ K|t-s|,
\end{align*}
where we took $\epsilon = |t-s|^{\frac{1}{2}}$. Since $w$ is bounded,
this implies H\"older 1/2 regularity in time. The spatial continuity follows from $ |w ( t,x ) - w ( t,y ) | \leq \| D^3 v  \|_{\infty} |x-y| $. In total, we get (recalling that $w = \partial_{x_{i}} \partial_{x_{j}} v$),
\begin{align*}
    | D^{2} v ( s,x ) - D^{2} v ( t,y ) | \leq C (  |t-s|^{\frac{1}{2}} +  |x-y| ),
\end{align*}
where $C >0$ is only dependent on $T_{0}$, 
$\sigma$, $T$, $d$, the quantities in \ref{L1}--\ref{B3},
and $\|D^kv\|_\infty$ for $k=0,1,2,3$. 
The results for $v$ and  $Dv$ follow by simpler similar arguments. Since $v$, $Dv$ and  $D^{2} v $ are uniformly
continuous, by Taylor's theorem (as in the proof Lemma
\ref{L_p_bounds}) $\mathcal{L} v $ is uniformly continuous with a
modulus only depending on the moduli of $v$, $Dv$ and  $D^{2} v$. 
\smallskip

\noindent (c) By \ref{B5} and the results from (b),  $\partial_{t} v \in UC ( ( 0,T_{0} ) \times \R^d  )$ 
by the equation \eqref{fractal_HJ}.  
\end{proof}

\subsection*{Global regularity and proofs of Theorem \ref{thm_HJ} and
  \ref{HJ_Holder_thm}} \label{sec:glob_reg}
From the local in time estimates, we construct a classical
solution $u$ of \eqref{fractal_HJ} on the whole interval $\left( 0,T
\right) \times \R^d$. By Theorem \ref{existence_visc_sol}, there is a
unique  viscosity solution $u$ of
\eqref{fractal_HJ} on $(0,T)$. To 
show that this solution is smooth, we proceed in steps.

\medskip
\noindent 1)\quad By Lemma
\ref{duhamel_regularity} (b) we find a $T_0 >0$ and a 
unique solution $v$ of \eqref{duhamel_integral} satisfying
\begin{align*}
    v, Dv, D^2v, D^3 v \in C_b ( [ 0,T_0 ] \times \R^d ) \text{ and } v ( 0 ) = u_{0},
\end{align*}
and by Corollary \ref{PDE_soln}, $v$ is a classical solution of \eqref{fractal_HJ} on $ (0 ,T_0 )$. 
Since classical solutions are viscosity solutions, $v$ coincides
with the unique viscosity solution $u$ on $(0,T_0)$.

\medskip
\noindent 2)\quad Fix $t_0 \in \left[ 0,T \right)$ and take the 
  value of the viscosity solution $u$ of (\ref{fractal_HJ}) as initial
  condition for \eqref{duhamel_integral} at $t=t_0$. Then $v(t_0,x)= u (
  t_0,x)$  and by Lemma \ref{existence_visc_sol},
  \begin{align}
    \|v \left( t_0, \cdot \right) \|_{W^{1,\infty} \left( \R^d \right)} \leq M_T.
    \label{indep_of_t0}
  \end{align}
  We apply Lemma \ref{duhamel_regularity} (a) with $k=2$ (translate
  time $t\to t-t_0$, apply the theorem, and translate back) to
  obtain a $T_1 >0$, independent of $t_0$, such that on
  \begin{align*}
    ( t_0, t_0 + T_1 ),
  \end{align*}
  we have a unique solution $v$ of \eqref{duhamel_integral} satisfying
 $v, \nabla v, (t-t_0)^{1/\sigma} D^2 v \in C_b$. Then 
  \begin{align*}
    v, \nabla v, D^2 v \in C_b \left( ( t_0 + \delta_1, t_0 + T_1 ) \times \R^d \right)
  \end{align*}
  for any $\delta_1\in(0, T_1)$. Let $\delta_{1} \leq \frac{1}{4} \min ( T_{0}, T_{1} ) $, and take $v \left( t_0 +\delta_1, \cdot \right)$ as initial
condition. By Lemma \ref{duhamel_regularity} (a)
again we find a $T_2>0$ such that on the interval
  \begin{align*}
    ( t_0 + \delta_1, t_0+ \delta_1 + T_2 )
  \end{align*}
  there exists a unique solution $v$ of \eqref{duhamel_integral} such
  that  for any $\delta_2\in(0,T_2)$,
$$v, \nabla v, D^2 v, t^{1/\sigma} D^3 v \in C_b(( t_0 +
  \delta_1+\delta_2, t_0+ \delta_1 + T_2 )).$$
  We define $\tilde{T} := \min ( T_{0}, T_{1}, T_{2} ) $, and
  let $\delta_{2} \leq \frac{1}{8} \tilde{T}$.
  Defining 
  $\delta := \delta_{1} + \delta_{2} \leq \frac{1}{2} \tilde{T} $, we find that
  $$v, Dv, \ldots, D^{3} v \in C_{b} ( ( t_{0} + \delta, t_{0}+ \delta + \tilde{T} ) \times \R^d  ).$$
  By Proposition \ref{time_proposition}, $\partial_{t} v \in C_{b}$,
  and $v$ is a classical solution of
  \eqref{fractal_HJ} on $ ( t_{0} + \delta, t_{0} +\delta+ \tilde{T} )$, therefore coinciding with $u$
  on this time interval. Note that $\tilde{T}>0 $ can be chosen
  independently of $t_{0}$ by \ref{A3}, \ref{B3}, \ref{B4}, and
  \eqref{indep_of_t0}.
  \medskip
  
\noindent 3)\quad We cover all of $(0,T)$ by intervals from step 1)
and 2), repeatedly taking $t_{0} = 0$, $\frac{1}{2} \tilde{T}$, $\tilde{T}$,
  $\frac{3}{2} \tilde{T}$, $\ldots$,  $\frac{N-1}{2} \tilde{T}$ with
  $\frac{N}{2} \tilde{T} \geq T$. We then find that the viscosity
  solution $u$ is a classical solution with 
  bounded derivatives on $(0,T)$ and the proof of Theorem \ref{thm_HJ}
  is complete.

\medskip  
\noindent 4) Theorem \ref{HJ_Holder_thm} follows from Theorem
\ref{thm_HJ} and Proposition \ref{time_proposition}.

\section{Fractional Fokker-Planck equations}
\label{sec:fracFP}
Here we prove the existence of smooth solutions of the fractional
Fokker-Planck equation, along with $C_b$, $L^1$, tightness, and time
equicontinuity in $L^1$ a priori estimates. The equation is given by
\begin{align}
  \left\{ \begin{array}{ll}
      \partial_t m - \mathcal{L}^* m + div \left( b(t,x)m \right) = 0 & \text{ in } \left( 0,T \right) \times \R^d, \\
  m \left( 0,x \right) = m_0 \left( x \right) & \text{ in } \R^d,
   \end{array}
\right.
  \label{Fokker_Planck}
\end{align}
where $b: \left[ 0,T \right] \times \R^d \to \R^d$, and $\mathcal{L}$
(and hence also $\mathcal{L}^*$) satisfies \ref{L1},\ref{L2}.

We first show preservation of positivity and a first $C_b$-bound for bounded
solutions.
\begin{proposition} \label{fp_comparison_principle}
    Assume \ref{L1} and $b, Db\in C_b((0,T)\times\R^d)$ and $m$ is a
    bounded classical solution of \eqref{Fokker_Planck}.
    \medskip
    
 \noindent (a) If $m_0\geq 0$, then $m(x,t)\geq 0$ for
 $(x,t)\in[0,T]\times \R^d$.
\medskip
 
\noindent (b) If $m_0\in C_b(\R^d)$, then
      $\|m(t,\cdot)\|_{\infty} \leq e^{\|(\mathrm{div}\,
        b)^{+}\|_\infty t}\|m_0\|_{\infty}$.
  \end{proposition}

In fact this result also holds for bounded viscosity solutions, but this is
not needed here. The result is an immediate consequence of the following lemma.

\begin{lemma}   \label{comparison_principle_fp_lem}
    Assume \ref{L1} and $b, Db\in C_b((0,T)\times\R^d)$ and $m$ is a
    bounded classical subsolution of \eqref{Fokker_Planck}. Then for
    $t\in[0,T]$, 
    \begin{align}    
      \|m(t,\cdot)^{+}\|_{\infty} \leq e^{\|(\mathrm{div}\,
        b)^{+}\|_\infty t}\|m_0^{+}\|_{\infty}
    \end{align}
\end{lemma}

\begin{proof}[Proof of Proposition \ref{fp_comparison_principle}]
(a) Apply Lemma \ref{comparison_principle_fp_lem} on $-m$ (which still
  is a solution) and note that $(-m_0)^+=0$. (b) Apply Lemma \ref{comparison_principle_fp_lem} on $m$ and
$-m$.   
\end{proof}
\begin{proof}[Proof of Lemma \ref{comparison_principle_fp_lem}]
    In non-divergence form we get (the
    linear!) inequality
    \begin{align*}
     \partial_t m - \mathcal{L}^* m + b\cdot Dm  +
     (\mathrm{div}\, b)\, m  \leq 0,
 \end{align*}
with $C_b$ coefficients by the assumptions. The proof is then completely standard and we only sketch the case that $\mathrm{div}\,
        b<0$. Let 
$$a=\sup_{(x,t)\in Q_T}m(x,t)^+-\|m_0^+\|_{\infty},$$
$\chi_R(x)=\chi(\frac xR)$ where $0\leq
\chi\in C^2_c$ such that $\chi=1$ in $B_1$ and $=0$ in $B_2^c$, and 
$$\Psi(x,t)=m(x,t)-\|m^+\|_{\infty}-at-\|m^+\|_{\infty}\chi_R(x).$$
We must show that $a\leq0$. Assume by
contradiction that $a>0$. Then there
exists a max point $(\bar x,\bar t)$ of $\Psi$ such that $\bar t>0$. At this max
point $m>0$ (since $a>0$) and
$$m_t\geq a,\quad Dm = D\chi_R, \quad\text{and} \quad  \mathcal{L}^* m \leq
\mathcal{L}^*\chi_R.$$
Hence using the subsolution inequality at this point and $\mathrm{div}\,
        b<0$, we find that
\begin{align*}
  a\leq m_t\leq \mathcal{L}^*m + b\cdot Dm  +
   (\mathrm{div}\, 
        b)\,m &\leq \|m^+\|_{\infty}\Big(\mathcal{L}^*\chi_R + b\cdot
        D\chi_R\Big).
  \end{align*}
An easy computation shows that that all $\chi_R$-terms converge to
zero as $R\to\infty$. Hence we pass to the
limit and find that $a\leq0$, a contradiction to $a>0$. The
result follows.
\end{proof}

The Fokker-Planck equation \eqref{Fokker_Planck} is mass and
positivity preserving (it preserves pdfs) and therefore may preserve
the $L^1$-norm in time. We will now prove a sequence of a priori
estimates for $L^1$ solutions of \eqref{Fokker_Planck}, using a ``very
weak formulation'' of the equation.

\begin{lemma}\label{lem:wf} Assume \ref{L1}, $m_0\in L^1_{\mathrm{loc}}$, $b,Db\in C_b$, and
  $m$ is a classical solution of \eqref{Fokker_Planck} such that
$m,Dm,D^2m\in C_b$. Then
  for every $\phi\in C_c^\infty(Q_T)$, $0\leq s<t\leq T$,
\begin{align}\label{eq:wf}\int_{\R^d} m \phi(x,t)\, dx= \int_{\R^d} m\phi(x,s)\, dx + \int_s^t \int_{\R^d} 
    m\big(\phi_t+\mathcal{L}\phi-b\cdot D\phi\big)(x,r)\, dx\, dr.
  \end{align}
  \end{lemma}
\begin{proof}
Note that $\mathcal L \phi\in C([0,T];L^1(\R^d))$ by Lemma
\ref{L_p_bounds_eq} with $p=1$. Multiply \eqref{Fokker_Planck}
by $\phi$, integrate in time and space, and integrate by parts. The
proof is completely standard, after noting that $\int \mathcal{L}^*m\,
\phi\, dx = \int m\,\mathcal{L}\phi\, dx$ in view of the assumptions of
the Lemma.  
  \end{proof}

\begin{remark} \label{rmk:bounded_testfunctions}
   If in addition $m\in C([0,T];L^1(\R^d))$, then a density argument shows
  that \eqref{eq:wf} holds for any $\phi\in C^\infty_b$.
\end{remark}  

Next we prove mass preservation, time-equicontinuity, and tightness
for positive solutions in $L^1$. For tightness  we need the
following result: 

 \begin{proposition}  \label{prop:tail-control-function}
     Assume \ref{L1} and $m_0\in P(\R^d)$. There exists a function $0\leq \psi \in C^{2}(\R^d)$ with $\|D\psi\|_\infty,$ $\|D^2\psi\|_\infty < \infty$, and  $\displaystyle \lim_{|x|\rightarrow \infty} \psi(x) = \infty$, such that 
\begin{align}\label{eq:tail-control-function}
    \int_{\R^{d}} \Psi(x) \, m_0(dx)<\infty , \quad \int_{ \R^{d} \setminus B_{1}} \Psi(x) \, \mu (dx)<\infty
\end{align}
\end{proposition}

\begin{proof}
    We let $\mu_{0} = \frac{\mu ( dx ) \mathbf{1}_{|x| \geq 1}}{\int_{B_1^c} \mu (dx)} $ and
$\Pi = \{ m_{0}, \mu_0  \} $ and apply 
    \cite[Proposition 3.8]{Espen-Indra-Milosz-2020}.
\end{proof}

\begin{proposition}\label{prop:L1}
Assume \ref{L1}, $m_0\in C_b$, $b,Db\in C_b$, and
  $m$ is a classical solution of \eqref{Fokker_Planck} such that
$m,Dm,D^2m\in C_b$. We also assume $m\in C([0,T];L^1(\R^1))$,  $m_0\geq 0$, and
$\int_{\R^d}m_0\,dx=1$.

\smallskip
\noindent (a)\ $m\geq 0$ and $\int_{\R^d}m(x,t)\,dx=1$
for $t\in[0,T]$.

\smallskip
\noindent (b)\ There exists a constant $c_0>0$ such that
  \begin{align*}    
    d_0 ( m ( t ), m ( s ) ) \leq c_0 ( 1+\|b\|_{\infty} )
    |t-s|^{\frac{1}{\sigma}} \qquad  \forall s,t \in [ 0,T ].
  \end{align*}
   
  \noindent (c)\ For $\psi$ defined in Proposition \ref{prop:tail-control-function} there is $c > 0 $ such
 that for $t\in[0,T]$,
 \begin{align}
   &\int_{\R^d}m(x,t)\psi(|x|)\,dx\leq
   \int_{\R^d}m_0\psi(|x|)\,dx \label{ti-eq}\\
   &\qquad+2\|\psi'\|_{C_b}+cT\|\psi'\|_{C_b}\Big(\|b\|_{C_b}+\int_{|z|<1}|z|^2d\mu(z)\Big)+T\int_{|z|>1}\psi(|z|)\,d\mu(z).\nonumber
   \end{align}
\end{proposition}

\begin{proof}
(a) \  By Proposition \ref{fp_comparison_principle}, $m\geq0$. Let
  $R>1$ and $\chi_R(x)=\chi(\frac xR)$
  for $\chi\in C^\infty_c$ such that $0\leq\chi\leq1$ and $\chi=1$ in
  $B_1$ and $=0$ 
  in $B_2^c$. We will apply Lemma \ref{lem:wf} with $\phi=\chi_R$ and
  $s=0$ and pass to the limit as $R\to\infty$. To do that, we write
  $\mathcal L=\mathcal L_1 +\mathcal L^1 = \int_{|z|<1}\cdots +\int_{|z|>1}\cdots$,
  and note that by Lemma \ref{L_p_bounds_eq} with $p=\infty$ and $\mu(B_1^c)=0$,
  $$\|\mathcal L_1 \chi_R\|_{C_b}\leq
  C\inf_{r\in(0,1)}\Big(r^{2-\sigma}\frac1{R^2}\|D^2\chi\|_{C_b}+(r^{1-\sigma}-1)\frac1R\|D\chi\|_{C_b}\Big)\leq C\frac1{R^2}\|\chi\|_{C^2_b},$$
  and then
  $$\|\partial_t\chi_R+\mathcal L_1 \chi_R - b\cdot D\chi_R\|_{C_b}\leq
  \frac1R \Big( \|\mathcal 
  \chi\|_{C^2_b}+\|b\|_{C_b}\|D\chi\|_{C_b}\Big)\stackrel[R\to\infty]{}{\longrightarrow}0.$$
Also note that $\|\mathcal L^1 \phi_R\|_{C_b}\leq 2\mu(B_1^c)$ and $\mathcal
L^1 \phi_R(x)\to 0$ for every $x\in\R^d$. Since $m\in
  C([0,T];L^1)$ by assumption, it follows by the dominated convergence
theorem that,
$$\int_0^t\int_{R^d}m\,\mathcal L^1\chi_R\, dx\,dr \stackrel[R\to\infty]{}{\longrightarrow}0.$$
Now we apply Lemma \ref{lem:wf} with $\phi=\chi_R$ and
  $s=0$ and pass to the limit in \eqref{eq:wf}
  as $R\to \infty$:
  $$\lim_{R\to\infty} \int_{\R^d} m(x,t) \chi_R(x)\, dx=
  \lim_{R\to\infty} \int_{\R^d} m_0\chi_R(x)\, dx+0.$$
  The result now
  follows from the dominated convergence theorem since $\chi_R\to1$
  pointwise and $\int m_0\,dx=1$.
  \smallskip

  \noindent (b) \  Fix a
  $\mathrm{Lip}_{1,1}$ function $\phi(x)$. For  $\epsilon\in(0,1)$, let
  $\phi_\epsilon\in C_b^\infty$ be an approximation (e.g. by mollification) such that
\begin{align}\label{mol-bnds} \|\phi-\phi_\epsilon\|_{C_b}\leq \epsilon \|D\phi\|_{C_b} \qquad
  \text{and}\qquad\|D^k\phi_\epsilon\|_{C_b}\leq
  c\epsilon^{(k-1)^+}\|\phi\|_{C^1_b}, \quad k\geq 0.
\end{align}
Applying Lemma \ref{lem:wf} and Remark \ref{rmk:bounded_testfunctions} 
with $\phi=\phi_\epsilon(x)$, then leads to
  \begin{align*}
    \int_{\R^d} (m(x,t)-m(x,s)) \phi_\epsilon(x)\, dx= \int_s^t \int_{\R^d} 
    m\big(0+\mathcal{L}\phi_\epsilon-b\cdot D\phi_\epsilon\big)(x,r)\, dx\, dr.
  \end{align*}
  By Lemma \ref{L_p_bounds_eq} with $p=\infty$ and \eqref{mol-bnds},
  \begin{align*}\|\mathcal{L}\phi_\epsilon\|_{C_b} & \leq c
    \inf_{r\in(0,1)}\Big(r^{2-\sigma}\|D^2\phi_\epsilon\|_{C_b}+r^{1-\sigma}\|D\phi_\epsilon\|_{C_b}+\|\phi_\epsilon\|_{C_b}\Big)\\
    &\leq c
  \inf_{r\in(0,1)}\big(r^{2-\sigma}\frac1\epsilon+r^{1-\sigma}+1\big)\|\phi\|_{C_b^1}\leq  C \epsilon^{1-\sigma}\|\phi\|_{C_b^1},
  \end{align*}
  and hence
 \begin{align*}
    \int_{\R^d} (m(x,t)-m(x,s)) \phi_\epsilon(x)\, dx\leq
    C|t-s|\epsilon^{1-\sigma}(1+\|b\|_{C_b})\|\phi\|_{C_b^1}\|m\|_{C(0,T;L^1)}.
 \end{align*}
 Then by adding and subtracting $(m(x,t)-m(x,s)) \phi_\epsilon(x)$ terms, we find that
 \begin{align*}
   & \int_{\R^d} (m(x,t)-m(x,s)) \phi(x)\, dx \\ & \leq \int_{\R^d}
   (m(x,t)-m(x,s)) \phi_\epsilon(x)\, dx + 2
   \|m\|_{C(0,T;L^1)}\|\phi-\phi_\epsilon\|_{C_b}\\
&\leq 
   C(|t-s|\epsilon^{1-\sigma} + \epsilon)(1+\|b\|_{C_b})\|\phi\|_{C^1_b}\|m\|_{C(0,T;L^1)}.
 \end{align*}
 Since $\|m\|_{C(0,T;L^1)}=1$ by part (a), and $\|\phi\|_{C^1_b}\leq
 2$ for $\mathrm{Lip}_{1,1}$-functions, the result follows from the
 definition of the $d_0$ distance in \eqref{d0} after a
 minimization in $\epsilon$.
 \smallskip

\noindent (c) \ Let $\psi_R(r)=\rho_1*(\psi \wedge R)(r)$ for
$r\geq1$, where $0\leq\rho_1\in C_c^{\infty}((-1,1))$ is symmetric and
has $\int \rho_1\,dx =1$ (a mollifier). We note that $\rho_1*\psi\leq
\psi$ and that $\psi \wedge R$ is
nondecreasing, concave, and $\nearrow \psi$. Standard arguments then show that $\psi_R \in C^\infty_b([1,\infty))$,
\begin{align}
  & 0\leq \psi_{R}\leq R,\quad 0\leq \psi_R'\leq \psi', \quad \psi_R''\leq 0,\quad\|\psi_R''\|_{C_b}\leq \|\rho_1'\|_{L^1}\|\psi'\|_{C_b},\label{deriv-b} \\
  & \psi_{R}\nearrow \rho_1*\psi\ (\leq \psi)
  \qquad\text{as}\qquad
  R\to\infty.\label{mon-conv}
\end{align}
The convergence as $R\to\infty$ is pointwise. We apply Lemma \ref{lem:wf} and remark \ref{rmk:bounded_testfunctions} with
$$\phi(x,t)=\phi_R(x):=\psi_R(\sqrt{1+|x|^2}).$$
Let $\LL=\LL_1+\LL^1$
as in the proof of part (a), and note that (using also \eqref{deriv-b}
and Lemma
\ref{L_p_bounds_eq} with $r=1$),
$$\|D\phi_R\|_{C_b}\leq c\|\psi'\|_{C_b},\ \|D^2\phi_R\|_{C_b}\leq
c\|\rho_1'\|_{L^1}\|\psi'\|_{C_b}, \ \|\LL_1\phi_R\|\leq
c\|\rho_1'\|_{L^1}\|\psi'\|_{C_b}\int_{|z|<1}|z|^2\,d\mu.$$
Next since $\psi_R$ is nonnegative, nondecreasing, and
subadditive\footnote{Nonnegative concave functions $h$ on $[0,\infty)$
    are subadditive: $h(a+b)\leq
h(a)+h(b)$ for $a,b\geq 0$.}, we observe that
  \begin{align*}
    |\psi_R(r)-\psi_R(s)|\leq \psi_R(r-s)
  \quad\text{for all}\quad r,s\geq0.
  \end{align*}
Hence we find that
\begin{align*}|\LL^1\phi_R(x)|\leq\ &\int_{|z|>1}\Big|\psi_R(\sqrt{1+|x+z|^2})-\psi_R(\sqrt{1+|x|^2})\Big|\,d\mu(z)\\
\leq\ &  \int_{|z|>1}\psi_R(|z|)\,d\mu(z)\leq \int_{|z|>1}\psi(|z|)\,d\mu(z).  
\end{align*}
From the estimates above we conclude that
\begin{align*}
  &\Big|\partial_t\phi_R+\mathcal{L}\phi_R-b\cdot
  D\phi_R\Big|\leq
  c\|\psi'\|_{C_b}\Big(\|b\|_{C_b}+\|\rho_1'\|_{C_b}\int_{|z|<1}|z|^2\,d\mu\Big)+\int_{|z|>1}\psi(|z|)\,d\mu.
\end{align*}
Inserting this estimate into \eqref{eq:wf} with $\phi=\phi_R$, along with $m\geq0$, $\int
m(x,t)\,dx=1$ (by part (a)), and $\phi_R(x)\leq\psi(\sqrt{1+|x|^2})$, we get 
\begin{align*}
  &\int_{\R^d} m(x,t) \phi_R(x)\,dx \leq \int_{\R^d} m_0(x) \psi(\sqrt{1+|x|^2})\,dx \\
  &\qquad+ T
c\|\psi'\|_{C_b}\Big(\|b\|_{C_b}+\|\rho_1'\|_{C_b}\int_{|z|<1}|z|^2\,d\mu\Big)+T
\int_{|z|>1}\psi(|z|)\,d\mu.
  \end{align*}
By the monotone convergence theorem and \eqref{mon-conv},
$$\lim_{R\to\infty}\int_{\R^d} m(x,t) \phi_R(x)\,dx=\int_{\R^d}
m(x,t)\, \rho_1*\psi(\sqrt{1+|x|^2})\,dx.$$ 
To conclude that \eqref{ti-eq} holds, we note that $\rho_1*\psi\geq
\psi-\|\psi'\|_{C_b}$
and $$\psi(|x|)\leq\psi(\sqrt{1+|x|^2})\leq\psi(|x|)+\|\psi'\|_{C_b}.$$
The proof of (c) is complete.
\end{proof}

Solutions in $L^1$ also have a better $C_b$ bound than the one in
Proposition \ref{fp_comparison_principle}. This bound is needed in the
local coupling case -- see Section \ref{sec:loc}.

\begin{lemma} \label{L-infinity_estimate}
Assume \ref{L1}, \ref{L2} (ii), $b\in C_b$, $0\leq m_0\in C_b$, and $0\leq m\in C_b(Q_T)$
is a classical solution of \eqref{Fokker_Planck}. If $m\in
C(0,T;L^1(\R^d))$, then there exist a constant $C >0$ only dependent
on $d, q, \sigma, T$, such that for any $1< p < p_0:= \frac{d}{d+1-\sigma}$,
    \begin{align*}
   \| m  \|_{C_b} \leq 1 \vee \Big[  \| m_{0}  \|_{C_b} + C
     T^{\frac{d - p(1+d- \sigma) }{p \sigma}}  \| b  \|_{C_b}
     \Big]^{\frac p{p-1}} .
    \end{align*}
\end{lemma}

\begin{proof} (Inspired by \cite[Proposition 2.2]{bardi2018uniqueness}) For any $y\in\R^d$, let  $ \phi ( s,x ) = K ( t-s, y-x)$ where
  $K$ is the heat kernel of Section \ref{sec:fracHK}. Then $\phi\geq0$
is smooth, $\int_{\R^d} \phi(x,s) dx = 1$, and $\phi$
    solves the backward heat equation 
    \begin{align} \label{backward_heat}
     \begin{cases}
    - \partial_{t} \phi - \mathcal{L} \phi = 0, \quad s<t, \\
    \phi ( x,t ) = \delta_{y} ( x ),
     \end{cases}
    \end{align}
    where the $\delta$-measure $\delta_{y}$ has support in $y$.
 Multiply \eqref{Fokker_Planck}
by $\phi$, integrate in time and space, and integrate by parts to get
\begin{align*}
  &\int m\phi(x,t) dx- \int m\phi(x-y,0) dx= \int_0^t \int 
  m(x,s)[\phi_t+\mathcal{L}\phi-b\cdot D\phi](x-y,s)\, dx\, ds
\end{align*}
or
\begin{align*}
  & m(y,t)= m * K(\cdot,t)(y)+ \int_0^t \int 
    (bm)(\cdot,s)* DK(\cdot,t-s)\, dx\, ds.
  \end{align*}
Then by the heat kernel estimates of \ref{L2} (ii), $\| DK ( s,\cdot )  \|_{L^p} \leq
        C s^{\frac{d - p(1+d)}{p \sigma}}$, the Hölder and
        Young's inequalities, the properties of $K$, and $\| m(\cdot,t)\|_{L^{q}}^{q}\leq \|m\|_{C_b}^{q-1}\|m(\cdot,t)\|_{L^1}=\|m\|_{C_b}^{q-1}$,
    \begin{align*}
         |m (y,t )| &\leq \| m_{0}  \|_{C_b} + \| b  \|_{C_b}\int_{0}^{t} \|D K(\cdot,t-s)\|_{L^p} \|m(\cdot,s)\|_{L^{p'}} \,dt \\
        &\leq \| m_{0}  \|_{\infty} + Ct^{\frac{d - p(1+d) + p \sigma }{p \sigma}} \| b  \|_{C_b} \|m\|_{C_b}^{1-\frac1{p'}},
    \end{align*}
    for $1\leq p\leq \frac{d}{1+d-\sigma}$ where $\frac{1}{p} + \frac{1}{p'} = 1$. Since $y$ is arbitrary, we get after taking the supremum and dividing both
    sides by $\| m  \|_{C_b}^{\frac{1}{p}} $ that 
    \begin{align*}
    \| m  \|_{C_b} \leq 1 \vee \big[  \| m_{0}  \|_{C_b} + C T^{\frac{d - p(1+d- \sigma) }{p \sigma}}  \| b  \|_{C_b}   \big]^{p'} .
    \end{align*}
    This concludes the proof.
\end{proof}

Finally, we state the main result of this section, the existence of
classical solutions of \eqref{Fokker_Planck} that are positive and
mass-preserving. 

\begin{proposition}
   \label{fokker_planck_holder}
        Assume \ref{L1}, \ref{L2}, $b, Db, D^2 b\in C_b\big( ( 0,T ) \times
    \R^d \big)$, $0\leq m_0 \in C_b^2 (\R^d)$, and $\int_{\R^d}m_0\,dx=1$. 
    \medskip
    
    \noindent (a) There exists a
    unique classical solution $m$ of (\ref{Fokker_Planck})
    satisfying $m\geq0$, $\int_{\R^d}m(x,t)\,dx=1$ for
    $t\in[0,T]$, and 
  \begin{align*}
    \|m \|_{L^{\infty}} + \|D m \|_{L^{\infty}}  + \|D^2 m \|_{L^{\infty}}+\|\partial_{t} m \|_{L^{\infty}}  \leq c,
  \end{align*}
  where $c$ is a constant depending only on $\sigma$, $T$, $d$, and
  $\|D^kb\|_{\infty}$ for $k=0,1,2$.
    \medskip
    
    \noindent (b) There exists a modulus $\tilde{\omega}$ only
    depending on $\|D^k m  \|_{\infty}, \| D^kb  \|_{\infty}$ for
    $k=0,1,2$, and \ref{L1}, such that for $s,t\in[0,T]$ and $x,y\in\R^d$,
    \begin{align*}
        | m  ( t,x ) - m ( s,y ) | + |Dm ( t,x ) - Dm ( s,y ) | \leq \tilde\omega ( |t-s| + |x-y| ).
    \end{align*}
    
    \noindent (c) If in addition $b,  Db \in UC ( ( 0,T ) \times \mathbb{R}^d )$, then there exists a modulus $\omega$ only depending on
     $\tilde{\omega}$, $\omega_b$, $\omega_{Db}$, $\|Db\|_{\infty}$, $m_0$, $T$, $\sigma$, and $d$, such that for $s,t\in[0,T]$ and $x,y\in\R^d$,
     $$|\mathcal L^* m(x,t)-\mathcal L^* m(s,y)|+|\partial_t m(x,t)-\partial_t m(s,y)| \leq \omega(|s-t|+|x-y|).$$
\end{proposition}

\begin{proof}
(a) The proof uses a Banach fixed point argument based on the Duhamel formula 
\begin{align}
    \label{duhamel_integral2}
      &m(t,x)=    \tilde\psi \left( m \right) \left( t,x \right) \\
      &:= K^* \left( t, \cdot \right) \ast m_0 \left( \cdot \right) \left( x \right) - \sum_{i=1}^d\int_0^t \partial_{x_i}K^* \left( t-s, \cdot \right) \ast ( b_im) \left(s,\cdot \right) ds,
\nonumber
\end{align}
and is similar to the proof of Theorem \ref{thm_HJ}. Here $K^*$ is the
heat kernel of $\mathcal L^*$. It is essentially a
corollary to Proposition 5.1 in \cite{droniou2003global} (but in our case the we have
more regular initial condition and hence no blowup of norms when $t\to0^+$).

Similar to the corresponding proof for the HJB equation,
we first show short-time $C^1$-regularity using the Duhamel formula.
Let $R_{0} = 1+  \| m_{0}  \|_{\infty} $, $R_{1} = (  2  +d \mathcal{K})
R_{0} + 1$, and the Banach (sub) space 
     \begin{align}
         X = \big\{ m: m, t^{1 / \sigma} Dm \in C_{b}\big( ( 0,T_{0} ) \times \mathbb{R}^{d} \big),  \ m\in C([0,T];
         L^1(\R^d) ),   \text{ and } \| m  \| \leq R_{1}\big\} , 
        \label{fp_banach_subspace}
    \end{align}
where $\| m  \| =  \|m\|_{C([0,T];
         L^1)}+ \| m  \|_{\infty} + \sum_{i=1}^{d} \| t^{1 / \sigma}
\partial_{x_{i}} m \|_{\infty} $. Then if $k ( \sigma )$ and $ \gamma
( \sigma )$ are defined in the proof of Proposition
\ref{duhamel_regularity} (a), we find from \eqref{duhamel_integral2}
that for $p\in\{1,\infty\}$, 
\begin{align*}
    \| \tilde{\psi} ( m ) ( t,x ) \|_{L^p} &\leq \| K^*  \|_{L^{1}} \| m_{0}  \|_{p}  + \sum_{i=1}^{d} \int_{0}^{t}  \| \partial_{x_{i}}   K^* ( t-s, \cdot ) \|_{L^{1}} \| b_{i}  \|_{\infty} \| m(s)  \|_{p}\,ds \\
    & \leq R_0  +  d k(\sigma) T_0^{1-\frac{1}{\sigma}} \|b   \|_{\infty} R_{1},
\end{align*}
and
\begin{align*}
  &|t^{1 / \sigma} \partial_{x_{j}} \tilde{\psi} ( m ) ( t,x ) | \\
  &\leq t^{1 / \sigma}\|\partial_{x_{j}} K^*  \|_{L^{1}} \| m_{0}  \|_{\infty}  + \sum_{i=1}^{d}  t^{1 / \sigma} \int_{0}^{t}  \|  \partial_{x_{i}}  K^* ( t-s, \cdot ) \|_{L^{1}} \| ( m \partial_{j} b_{i} + b_{i} \partial_{j} m )  \|_{\infty}  ds \\
    &\leq \mathcal{K} R_{0} + \sum_{i=1}^{d} t^{1 / \sigma}
  \int_{0}^{t} \mathcal{K} ( t-s )^{- 1 / \sigma} \Big[   \|  m
    \|_{\infty} \| \partial_{j} b_{i} \|_{\infty} + s^{-1 /
      \sigma}\|b_{i} \|_{\infty} \| s^{1 / \sigma}\partial_{j} m
    \|_{\infty} \Big] ds \\ 
    &\leq \mathcal{K} R_{0} + \Big[  k ( \sigma ) T_{0} \| D b \|_{\infty}  + \gamma ( \sigma ) T_{0}^{1 - 1 / \sigma} \| b \|_{\infty}  \Big]d R_{1},
\end{align*}
Computing the full norm, we get
\begin{align*}
    &\| \tilde{\psi} ( m )  \| \\
    &\leq ( 2+ d \mathcal{K} ) R_{0} + \underbrace{ \bigg[
      2 d
      k(\sigma) T_0^{1-\frac{1}{\sigma}} \|b   \|_{\infty}  +
      d^2 \Big[  k ( \sigma ) T_{0} \| D b  \|_{\infty}  + \gamma ( \sigma ) T_{0}^{1 - 1 / \sigma} \| b  \|_{\infty}  \Big] \bigg] R_{1}}_{:=c(T_{0})}.
\end{align*}
We take $T_{0}>0$ so small that $c ( T_{0} ) \leq 1 / 2$. Then it
follows  that $\tilde{\psi}$ maps $X$ into itself by the definition of
$R_{1}$. It is also a contraction since for $m_{1}, m_{2} \in X$, it
easily follows that
\begin{align*}
  &\| \tilde{\psi} ( m_{1} ) - \tilde{\psi} ( m_{2} ) \|
  \leq c ( T_{0} ) \| m_{1} - m_{2}  \|.
\end{align*}
An application of Banach's fixed point theorem in $X$ then concludes
the proof. Note that we only needed $m_{0} \in C_{b}$ and $b,Db \in
C_{b}$ to obtain the result.

We can now repeatedly differentiate
the Duhamel formula \eqref{duhamel_integral} and use similar
contraction arguments to conclude that if $b,Db,...,D^{k}b \in C_{b} ( (
0,T ) \times \mathbb{R}^d )$, then there exists a solution $m\in X$
such that
    \begin{align*}
         &D^2m, ..., D^{k-1} m, t^{\frac{1}{\sigma}} D^{k} m \in C_{b}
      ( ( 0,T_{0} ) \times \mathbb{R}^{d} )\quad \text{for\quad$T_{0} >0$ sufficiently small.}
    \end{align*}
    In a similar way as in Proposition \ref{time_proposition} (a) and
    Corollary \ref{PDE_soln} for the HJB-equation, it now follows
    that $m$ is a classical solution to \eqref{Fokker_Planck}. By Lemma
    \ref{fp_comparison_principle} and Lemma \ref{prop:L1} (a), we then
    have global in time bounds $m$ in $C_b\cap C([0,T];L^1)$. We
    can therefore extend the local existence and the derivative
    estimates to all of $[0,T]$. The argument is very similar to the
    proof in Section \ref{time_proposition} and we omit it. Finally,
    by Lemma \ref{prop:L1} (a) again, we get that $m\geq0$ and
    $\int_{\R^d}m(x,t)\,dx=1$.    
    \medskip
    
    \noindent (b) Part (b) follows in a similar way as part (b) in Theorem \ref{time_proposition}. 
    We omit the details.
    \medskip
     
    \noindent (c) From part (a), (b), and the assumptions, the
    function $g ( t,x ) = \text{div} ( m b )$ satisfies $g, \nabla g
    \in C_{b} ( ( 0,T ) \times \mathbb{R}^{d} )$ and $g \in UC ( ( 0,T
    ) \times \mathbb{R}^{d} )$. Lemma \ref{DH-lem} (b) (with $K^*$
    instead of $K$) then gives that
    $\partial_t \Phi(g),\mathcal{L}^{*} \Phi(g) \in UC ( ( 0,T ) \times
    \mathbb{R}^{d} )$ with modulus $\omega$ only dependent on
    $\sigma,T,d, \| g \|_{\infty} , \| \nabla g \|_{\infty}$ and $
    \omega_{g} $. A similar, but simpler argument shows that $\partial_t K_t^**m_0=\mathcal{L}^{*} K_t^**m_0 \in UC((0,T) \times
    \mathbb{R}^{d})$. Since $m=K_t^**m_0-\Phi(g)$, this concludes the proof. 
\end{proof}

\section{Existence for MFGs with nonlocal coupling -- Proof of Theorem \ref{existence_theorem}}
\label{sec:pf}
We adapt \cite{lionsCCDF,cardaliaguet2013mfg,ACDPS20} and use the
Schauder fixed point theorem. 
We work in  $C ( [ 0,T ], \mathbf{P}(\R^d ))$ with metric
$d(\mu, \nu ) = \sup_{t \in [ 0,T ]} d_0 ( \mu ( t ), \nu (t ) )$ and the subset 
\begin{align}
  \mathcal{C} := \left\{ \mu \in C ( [ 0,T ] , \mathbf{P} ( \R^d ) )
  : \sup_{t\in[0,T]}\int_{\R^d}\psi(|x|)\,\mu(dx,t)\leq C_{1}, \ \sup_{s \neq t} \frac{d_0 ( \mu ( s ) , \mu ( t ) )}{ |s-t|^{\frac{1}{\sigma}}} \leq C_2 \right\},
  \label{C_subset}
\end{align}
where $\psi$ is defined in Proposition \ref{prop:tail-control-function} and the
constants $C_1, C_{2} > 0$ are to be determined. 
For $\mu \in \mathcal{C}$, define $S ( \mu ) :=  m$ where $m$ is the classical solution of the fractional FPK equation
\begin{align}
   \left\{
\begin{aligned}
  \partial_t m - \mathcal{L}^{*} m - div \big( D_p H (x,u, Du )m \big) &= 0,   \\
    m (0, \cdot ) &= m_0 ( \cdot ),
\end{aligned}  
\right. 
\label{fokker-planck_existence}
\end{align}
and $u$ is the classical solution of the fractional HJB equation
\begin{align}
   \left\{
\begin{aligned}
  - \partial_t u - \mathcal{L} u + H (x,u, Du ) &= F (x,\mu ), \\
	 u (x,T ) &= G (x, \mu (T ) ).
\end{aligned}  
\right.
\label{hamilton-jacobi_existence}
\end{align}
Let $\mathcal{U} := \left\{ u: u \text{ solves
  (\ref{hamilton-jacobi_existence}) for } \mu \in \mathcal{C}
\right\}$ 
and
    $\mathcal{M} := \left\{ m: m \text{ solves (\ref{fokker-planck_existence}) for } u \in \mathcal{U} \right\}.$
\medskip

\noindent 1. ($\mathcal C$ convex, closed, compact). 
The subset
$\mathcal{C}$ is convex and closed in $C ( [ 0,T ] , \mathbf{P} (
\R^d ) )$ by standard arguments. 
It is compact by the Prokhorov and Arzèla-Ascoli
theorems. \medskip

\noindent 2. ($S:\mathcal C\to\mathcal C$ is well-defined).
By
\ref{L1}, \ref{L2}, \ref{A1}--\ref{A6}, 
Theorem \ref{thm_HJ} and \ref{HJ_Holder_thm}, 
there is a unique solution $u$ of 
(\ref{hamilton-jacobi_existence}) with
\begin{align} \label{U_bounds}
    & \|u\|_\infty ,\|Du\|_\infty,\cdots,\|D^3 u\|_\infty, \|\partial_t u\|_\infty \leq U_1, \\
\notag  &\partial_t u, u, Du, D^2 u, \mathcal L u \quad \text{equicontinuous with modulus } \omega ,
    \end{align}
where $U_1$ depends on $d, \sigma$ and the  spatial regularity of $F$,
$G$ and $H$. 
The modulus $\omega$ depends in addition 
on $C_2$ in \eqref{C_subset}. By the uniform bound in \ref{A2}, $U_1$
is independent of $\mu$.
By Proposition \ref{fokker_planck_holder} part (a)--(c), for any
$u\in\mathcal U$ 
there is a unique $m$ solving \eqref{fokker-planck_existence} such that
\begin{align} \label{M_bounds}
    & \|m\|_\infty ,\|Dm\|_\infty,\|D^2m\|_\infty, \|\partial_t m\|_\infty \leq M_1, \\
\notag    &\partial_t m, m, Dm, \mathcal L^* m\quad \text{are equicontinuous with modulus } \bar{\omega},
    \end{align}
where $M_1$ depends on $U_1$ and the local regularity of $H$ but not
on $\mu$. The modulus $\bar{\omega}$ depends in addition on $\omega$. By Lemma
\ref{prop:L1} (b)--(c),
\begin{align*}
     d_0 ( m ( s ), m ( t ) ) &\leq c_0 ( 1 + \|D_p H ( \cdot , Du ) \|_{\infty} ) |s-t|^{\frac{1}{\sigma}}, \\
   \int_{\R^d}m(x,t)\psi(|x|)\,dx &\leq
   \int_{\R^d}m_0\psi(|x|)\,dx \\
   \qquad+2\|\psi'\|_{C_b}+c&T\|\psi'\|_{C_b}\Big(\|D_{p} H ( \cdot, Du )\|_{C_b}+\int_{|z|<1}|z|^2d\mu(z)\Big)+T\int_{|z|>1}\psi(|z|)\,d\mu(z).
\end{align*}
By \eqref{U_bounds} and \ref{A3}, 
$ \| D_p H ( x , Du )  \|_\infty \leq \tilde{C}$, where $\tilde{C} $ 
is independent of $\mu$. Hence, we take 
$C_{1} = \int_{\R^d}m_0\psi(|x|)\,dx    
+2\|\psi'\|_{C_b}+cT\|\psi'\|_{C_b} \tilde{C} +\int_{|z|<1}|z|^2d\mu(z)\Big)+T\int_{|z|>1}\psi(|z|)\,d\mu(z) $, and
$ C_2 = c_0 ( 1+\tilde{C} )$ 
and get that $S$ maps $\mathcal{C}$ into itself. \medskip

\noindent 3. ($S$ is continuous). We use the well-known result:
\begin{lemma}  \label{Compact_convergence}
  Let $ ( X,d )$ a metric space, $K \subset \subset X$ a compact subset and $( x_n ) \subset K$ a sequence such that all convergent subsequences have the same limit $x^* \in K$. Then $x_n \rightarrow x^*$.
\end{lemma}
\noindent Define
$X_1:=\{f : f,Df,D^2f,\partial_{t} f,\mathcal
    L f \in C_b\}$ and $X_2:=\{f : f,Df,\partial_{t} f,\mathcal
L^* f \in C_b\}$,  equipped with the metric
of local uniform convergence, taken at all the derivatives.
Then $X_1$ and  $X_2$ are complete metric spaces.
By \eqref{U_bounds}, \eqref{M_bounds},
Arzela-Ascoli, and a diagonal (covering) argument 
$\mathcal{U}$ and  $\mathcal{M}$ are compact in $X_1$ and $X_2$, respectively. 

Let $ \mu_{n} \to \mu \in \mathcal{C}$,
and let $ ( u_{n}, m_{n} )$ be the corresponding solutions of
\eqref{hamilton-jacobi_existence} and \eqref{fokker-planck_existence}. 
Take a convergent subsequence $ ( u_{n} ) \supset u_{n_{k}} \to
\tilde{u} \in \mathcal{U}$ and let $\mathcal{L} = \mathcal{L}_1 +
\mathcal{L}^1 =  
\big(\int_{|z|< 1} + \int_{|z| \geq 1 } \big) ( \ldots )$. By uniform
convergence $\mathcal{L}_{1 } u_{n_{k}} (t,x  ) \to
\mathcal{L}_{1 } \tilde{u} (t,x)$, and by dominated convergence 
$\mathcal{L}^{1 } u_{n_{k}} (t,x  ) \to \mathcal{L}^{1 } \tilde{u} (t,x  ) $.
By \ref{A1}, \ref{A3} and for any $(t,x) \in ( 0,T ) \times \R^d $:
\begin{align*}
    &\big| - \partial_t \tilde{u} ( t,x ) - \mathcal{L} \tilde{u} ( t,x ) + H (x, D\tilde{u} ( t,x ) ) - F (x,\mu ) \big| \\
    &\leq \big| \partial_t u_{n_k} ( t,x ) - \partial_t \tilde{u}( t,x ) \big| + \big| \mathcal{L} u_{n_k} ( t,x ) - \mathcal{L} \tilde{u} ( t,x ) \big| \\ 
    &+ \big|H ( x, Du_{n_k} ) - H ( x, D \tilde{u} ) \big| + \big| F ( x, \mu_{n_k} ( t ) ) - F ( x, \mu ( t ) ) \big|   \\
    &\rightarrow 0,
\end{align*}
and $\big|\tilde{u} ( T,x ) - G ( x, \mu ( T ) ) \big| \leq | \tilde{u} ( T,x ) - u_{n_{k}} ( T,x )  |   + \big| G ( x, \mu_{n_k} ( T )
) - G ( x, \mu ( T ) ) \big| \rightarrow 0$.  
This shows that
$\tilde{u}$ solves \eqref{hamilton-jacobi_existence} with $\mu$ as input.
By uniqueness of the HJB equation, compactness of $\mathcal{U}$ in $X_{1}$, and Lemma \ref{Compact_convergence}, 
we conclude that $u_n \rightarrow u$ in $X_1$.

A similar argument shows that $m_n \to m \in  X_{2}$.
By compactness of $\mathcal C$ in part 2, uniqueness of
solutions, and Lemma \ref{Compact_convergence}, we also find that
 $m_n \rightarrow m$ in $C ( [ 0,T ] , \mathbf{P} ( \R^d ) )$. The map
$S:\mathcal C\to\mathcal C$ is therefore continuous. \medskip

 \noindent 4. (Fixed point).
 By Schauder fixed point theorem there then exists a fixed point $S ( m ) = m$,
 and this fixed point is a classical solution of
 \eqref{frac_MFG_system} and the proof of Theorem \ref{existence_theorem} is complete.

\section{Existence for MFGs with local coupling -- Proof of Theorem
  \ref{local_theorem}}
\label{sec:loc}

\noindent 1. (Approximation) We follow Lions
\cite{lionsCCDF,cardaliaguet2013mfg}, approximating by
a system with non-local coupling and passing to the limit.
Let $\epsilon>0$, $0 \leq \phi \in C_c^{\infty}$ with $\int_{\R^d}
\phi = 1$, $\phi_{\epsilon} := \frac{1}{\epsilon^d} \phi (
x/ \epsilon )$, and for $\mu \in P(\R^d)$ let
$F_{\epsilon} ( x, \mu ) := f ( x, \mu \ast \phi_{\epsilon} ( x )
).$
For each fixed $\epsilon >0$, $F_{\epsilon}$ is a nonlocal coupling
function satisfying \ref{A1}--\ref{A2}, since $\|D^\beta(\mu \ast
\phi_{\epsilon}) \|_{\infty} \leq \|\mu\|_1
\|D^\beta\phi_{\epsilon}\|_{\infty}=\|D^\beta\phi_{\epsilon}\|_{\infty}$. 
Assumptions \ref{L1}--\ref{L2}, \ref{A1}--\ref{A6} then hold for the approximate system
\begin{align}
  \begin{split}
    \left\{ \begin{array}{ll}
        -\partial_t u_{\epsilon} - \mathcal{L} u_{\epsilon} + H ( x,Du_{\epsilon} ) = F_{\epsilon} ( x, m_{\epsilon} ( t) ) & \text{ in } ( 0,T ) \times \R^d, \medskip \\
      \partial_t m_{\epsilon} - \mathcal{L}^{*} m_{\epsilon} - div ( m_{\epsilon} D_p H ( x, Du_{\epsilon}  ) ) = 0 & \text{ in } ( 0,T ) \times \R^d, \medskip \\
      m ( 0 ) = m_0 ,\qquad u ( x,T ) = g ( x ),
    \end{array}
    \right.
  \end{split}
  \label{loc_MFG_system_reg}
\end{align}
and by Theorem \ref{existence_theorem} there exist a classical solution $( u_{\epsilon}, m_{\epsilon} )$ of this system.
\medskip

\noindent 2. (Uniform bounds)
Since either \ref{A3'} or \ref{A2''} holds,
$F_\epsilon(x,m_\epsilon(t))$ is uniformly bounded in
$\epsilon$. In the case of \ref{A3'} this follows from Lemma
\ref{L-infinity_estimate} and the estimate
\begin{align} \label{meps_bound}
   \| m_{\epsilon}  \|_{C_b} \leq 1 \vee \Big[  \| m_{0}  \|_{C_b} + C
       T^{\frac{d - p(1+d- \sigma) }{p \sigma}} \| D_{p} H ( \cdot, Du_{\epsilon} )  \|_\infty
     \Big]^{\frac p{p-1}} \leq K
\end{align}
for $K$ independent of $\epsilon$.
By Theorem
\ref{existence_visc_sol} (b) and \ref{A3} we then have
\begin{align} \label{ueps_bound}
    \|u_{\epsilon}\|_\infty &\leq \|g\|_\infty 
    + ( T-t ) (  \|F_{\epsilon} ( \cdot, m_{\epsilon} ( t ) ) \|_\infty 
    + \| H ( \cdot,0 )  \|_\infty ) \leq \tilde{K}
\end{align}
for $\tilde{K} >0$ independent of $\epsilon$, and since
$F_{\epsilon}$ is also continuous, 
by Theorem \ref{IshiiLions}
\begin{align} \label{Dueps_bound}
\|Du_{\epsilon}\|_{\infty} \leq C
\end{align}
 for $C \geq 0$ independent of
$\epsilon$ ($C$ depends on  $F_\epsilon$ only through its
 $C_b$-norm). Under \ref{A3'}  $m$ is bounded and satisfies
\eqref{meps_bound}, and this is still true if \ref{A3'} is replaced by
\ref{A2''} in view of the uniform bound on $Du_\epsilon$ in \eqref{Dueps_bound}.
\medskip

\noindent 3. (Improvement of regularity) 
The Duhamel formulas for $m_{\epsilon}$ and 
$Du_{\epsilon}$ are given by
\begin{align}
    m_{\epsilon} ( t ) &= K^*_{\sigma} ( t ) \ast m_0 - \sum_{i=1}^d \int_0^t \partial_i K^*_{\sigma} ( t-s ) \ast m_{\epsilon} [ D_p H( \cdot ,Du_{\epsilon} ( s ) )]_i ds,
    \label{m_eqn}\\
    Du_{\epsilon} ( t ) &= K_{\sigma} ( t ) \ast Du_0 - \int_0^t D_x K_{\sigma} ( t-s ) \ast ( H ( \cdot,Du_{\epsilon} ( s ) ) - F_{\epsilon} ( \cdot, m_{\epsilon} ( s, \cdot ) ) ) ds.
    \label{du_eqn}
\end{align}
where $K_{\sigma} ( t ) = K_{\sigma} ( t,x )$ and $K^*_{\sigma} ( t )
= K^*_{\sigma} ( t,x )$ are the fractional heat kernels in $\R^d$
corresponding to $\mathcal{L}$ and $\mathcal{L}^*$. Fractional
differentiations of these will lead to improved regularity. 

Assume that for $k \in \{ 0,1,2\}$ and $\alpha \in [ 0,1 )$,
there is $C \geq 0$ independent of $\epsilon$ such that for all $t\in[0,T]$,
\begin{align} 
    \|m_{\epsilon}(t) \|_{C^{k,\alpha} ( \R^d ) } + \|Du_{\epsilon} (t)\|_{C^{k,\alpha} ( \R^d ) } \leq C.
    \label{starting_point}
\end{align}
We will show that for any $\delta\in(0,\alpha)$ and $s\in
(0,\sigma-1)$ there is $\tilde
C\geq0$ independent of $\epsilon$ and $t$ such that
\begin{align}\label{iteration}
  \begin{cases}
     \|m_{\epsilon}(t) \|_{C^{k,s+\alpha-\delta} ( \R^d ) } +
     \|Du_{\epsilon}(t) \|_{C^{k,s+\alpha-\delta} ( \R^d ) } \leq
     \tilde C, & \text{for } s + \alpha -  \delta \leq 1,\\
     \|m_{\epsilon}(t) \|_{C^{k+1,s+\alpha-\delta-1} ( \R^d ) } +
     \|Du_{\epsilon}(t) \|_{C^{k+1,s+\alpha-\delta-1} ( \R^d ) } \leq
     \tilde C, & \text{for } s + \alpha -  \delta > 1, 
  \end{cases}
  \end{align}

Assume first $\alpha \in ( 0,1 )$ and consider the $m_\epsilon$-estimate. When
\eqref{starting_point} holds, then $m_{\epsilon} D_p H
(x,Du_{\epsilon}  ) \in C^{k,\alpha} ( \R^d 
)$ by the chain rule and \ref{A3}, and $|D|^{\alpha-\delta} D^k
\big( m_{\epsilon} D_p H ( x,Du_{\epsilon}  ) \big) \in
C_{b}^{0,\delta} ( \R^d )$ for $ \delta \in ( 0, \alpha )$ by \cite[Proposition
  2.7]{silvestre2007regularity}.
Let $s \in ( 0,\sigma -1 )$ and 
apply $|D|^{s} |D|^{\alpha-\delta} D^k$ 
to \eqref{m_eqn},
\begin{align*}
    |D|^{s} |D|^{\alpha-\delta} D^k m_{\epsilon} &= K_\sigma^* ( t ) \ast |D|^{s + \alpha - \delta} D^k m_0 \\
                                                 &- \sum_{i=1}^{d} \int_0^t |D|^{s} DK_\sigma^* ( t-s ) \ast |D|^{\alpha-\delta} D^k \big[   m_{\epsilon} D_p H (   \cdot, Du_{\epsilon} ) \big]_{i} ds.
\end{align*}
By Young's inequality and Proposition \ref{kernel_estimates} (heat
kernel estimates), 
\begin{align*}
   \| |D|^{s + \alpha - \delta} D^k m_{\epsilon} \|_{\infty} \leq  \| |D|^{ s +\alpha- \delta}D^k m_0 \|_{\infty} 
   +c \frac{T^{1-\frac{1+s}{\sigma}} }{1-\frac{1+s}{\sigma}} \||D|^{\alpha-\delta}  D^k ( m_{\epsilon} D_p H ( \cdot , Du_{\epsilon} ) ) \|_{\infty} ,
\end{align*}
and taking $\delta<\alpha/2$, we get uniform in $\epsilon$ Hölder estimates by \cite[Proposition
  2.9]{silvestre2007regularity},  
\begin{align*}
    m_{\epsilon} ( t ) \in \begin{cases}
     C_{b}^{k, s + \alpha -  2\delta} ( \R^d ),  &\text{for } s + \alpha - 2 \delta \leq 1, \\
     C_{b}^{k+1, s + \alpha -  2\delta-1} ( \R^d ), & \text{for } s + \alpha - 2\delta > 1.
 \end{cases}
\end{align*}
The case $\alpha = 0$ follows in a similar but more direct way differentiating
\eqref{m_eqn} by $| D |^{s} D^{k} $ instead of $|D|^{s}
|D|^{\alpha-\delta} D^k$ as above. The estimates on $Du_{\epsilon}$ follow similarly.
\medskip

\noindent 4. (Iteration and compactness)
Starting from \eqref{meps_bound}, \eqref{ueps_bound}, and \ref{A2'}
and \ref{A3}, we iterate using \eqref{iteration} to find that 
$$\|u_{\epsilon} ( t ) \|_{C_b^3 (\R^d)} + \|m_{\epsilon} ( t )
\|_{C_b^2 (\R^d)} \leq C$$
independent of $\epsilon$ and $t\in[0,T]$. 
By Proposition \ref{time_proposition} and 
Proposition \ref{fokker_planck_holder}, we then find that
\begin{align*} 
  &  \|\partial_t u_{\epsilon}\|_{\infty} \leq U \quad\text{and}\quad
  \partial_t u_{\epsilon}, u_{\epsilon}, Du_{\epsilon}, D^2
  u_{\epsilon}, \mathcal L u_{\epsilon} \quad \text{equicontinuous
    with modulus } \omega , \\
  &  \|\partial_t m_{\epsilon}\|_{\infty} \leq M  \quad\text{and}\quad
\partial_t m_{\epsilon}, m_{\epsilon}, Dm_{\epsilon}, \mathcal L^* m_{\epsilon}\quad \text{equicontinuous with modulus } \bar{\omega},
    \end{align*}
where $U$, $\omega$, $M$ and $\bar{\omega} $ are independent of  $\epsilon$.
As in the proof of Theorem \ref{existence_theorem},
these bounds imply compactness of
$(m_{\epsilon}, u_{\epsilon})$ in $X_1\times X_2$ (see below Lemma
\ref{Compact_convergence} for 
the definitions).
\medskip

\noindent 5. (Passing to the limit) We extract a convergent subsequence, 
$( u_{\epsilon_k}, m_{\epsilon_k} )
\rightarrow ( u, m )$ in $X_1\times X_2$. By a direct calculation the
limit $( u,m )$ solves equation (\ref{loc_MFG_system}). This concludes
the proof of Theorem \ref{local_theorem}.

\appendix

\section{Uniqueness of solutions of MFGs -- Proof of Theorem
  \ref{theorem:uniqueness}} \label{sec:pf_uniq}
The proof of uniqueness 
is essentially the same as the proof in the College de
France lectures of P.-L. Lions \cite{lionsCCDF,cardaliaguet2013mfg}.
Let $\left( u_1, m_1 \right)$ and $\left( u_2, m_2 \right)$ be 
two classical solutions, and set 
$\tilde{u} = u_1 - u_2$ and $\tilde{m} = m_1 - m_2$. 
By (\ref{frac_MFG_system}) and integration by parts,

  \begin{align*}
    \frac{d}{dt} \int_{\R^d} \tilde{u} \tilde{m} \,dx &=  \int_{\R^d} \frac{\partial}{\partial t} \left( \tilde{u} \tilde{m} \right) \,dx = \int_{\R^d} \left( \partial_t \tilde{u} \right) \tilde{m} + \tilde{u} \left( \partial_t \tilde{m} \right) \,dx \\
    &= \int_{\R^d} \bigg[ \Big(  - \mathcal{L} \tilde{u} + H \left( x, Du_1 \right) - H \left( x, Du_2 \right) - F \left( x, m_1 \right) + F \left( x, m_2 \right) \Big) \tilde{m} \\ 
    & \ \ \ \ \ \ \ \ \quad + \tilde{u} \mathcal{L}^* \tilde{m} - \langle D \tilde{u}, m_1 D_p H \left( x, Du_1 \right) - m_2 D_p H \left( x, Du_2 \right) \rangle \bigg] \,dx.
  \end{align*}
By the definition of the adjoint,
 $ \int_{\R^d}  \left( \mathcal{L} \tilde{u} \right) \tilde{m} - \tilde{u} \left(  \mathcal{L}^* \tilde{m} \right) \,dx = 0,$
and from \ref{A7} we get
\begin{align*}
    \int_{\R^d} \left( -F \left( x, m_1 \right) + F \left( x, m_2 \right) \right) d \left( m_1 -m_2 \right) \left( x \right) &\geq 0 \qquad \forall m_1,m_2 \in P ( \R^d ).
\end{align*}
For the remaining terms on the right hand side, we use a Taylor expansion and \ref{A8},
\begin{align*}
    &\int_{\R^d} \bigg[- m_1 \Big( H \left( x,Du_1 \right) - H \left( x,Du_2 \right) - \langle D_p H \left( x,Du_1 \right), Du_2 - Du_1 \rangle \Big) \\ 
    &\qquad\ \,- m_2 \Big( H \left( x,Du_2 \right) - H \left( x,Du_1 \right) - \langle D_p H \left( x,Du_2 \right), Du_1 - Du_2 \rangle \Big) \bigg] \,dx \\
    &\leq - \int_{\R^d} \frac{m_1 + m_2}{2C} |Du_2 - Du_1|^2 \,dx.
\end{align*}
Integrating from $0$ to $T$, using the fact that $\tilde{m} \left( t=0 \right) = 0$ and $\tilde{u} \left( t=T \right) = G \left( x,m_1 \left( T \right) \right) - G \left( x,m_2 \left( T \right) \right)$,
\begin{align*}
    \int_0^T \frac{d}{dt} \int_{\R^d} \tilde{u} \tilde{m} \,dx \,dt =  \int_{\R^d} \left( G \left( x,m_1 \left( T \right) \right) - G \left( x,m_2 \left( T \right) \right) \right) \left( m_1 \left( x,T \right) - m_2 \left( x,T \right) \right) \,dx \geq 0,
\end{align*}
where we used \ref{A7} again. Combining all the estimates we find that
\begin{align*}
  0 \leq - \int_0^T \int_{\R^d} \frac{m_1+m_2}{2C} |Du_1 - Du_2|^2 \,dx \,dt
\end{align*}
Hence since the integrand is nonnegative it must be zero and $Du_1 =
Du_2$ on the set $\left\{ m_1 > 0 \right\} \cup \left\{ m_2 > 0
\right\}$. This means that $m_1$ and $m_2$ solve the same equation
(the divergence terms are the same) and hence are equal by uniqueness.
Then also $u_1$ and $u_2$ solve the same equation and $u_1=u_2$ by standard uniqueness for nonlocal HJB equations (see e.g. \cite{jakobsen2005continuous}).
The proof is complete.

\section{Proof of Lemma \ref{DH-lem}}\label{pf-DH-lem}

\noindent a) \ The proof is exactly the same as in \cite{imbert2005non}. The
difference is that $f$ only needs to be $C^1$ in space, since $D_x K$
is integrable in $t$.
\medskip

\noindent b) \ {\em Part 1:} Uniform continuity in $x$ for $\mathcal L
\Phi ( f )$ and $\partial_{t} \Phi ( f )$. By the definition of $\mathcal L$,
\begin{align*}
    &\mathcal{L} [ \Phi ( f ) ] ( t,x ) = \int_{0}^t \mathcal{L} K ( t-s, \cdot ) \ast f(s, \cdot) ( x ) ds \\
    &= \int_{0}^t \int_{\mathbb{R}^d} \Big[\int_{\mathbb{R}^{d}} K ( t-s,
  y+z ) - K ( t-s, y ) - \nabla_{x} K ( t-s,y ) \cdot z 1_{|z|<1} d
  \mu ( z )\Big] f ( s, x - y ) dy ds \\
  &= \int_{0}^t \int_{\mathbb{R}^d}\int_{|z|<1} \Big(\cdots\Big) + \int_{0}^t \int_{\mathbb{R}^d}\int_{|z|>1} \Big(\cdots\Big)=:I_1(t,x)+I_2(t,x).
\end{align*}
After a change of variables and $\|K(t,\cdot)\|_{L^1}=1$,
\begin{align*}
   |I_2(t,x_1)-I_2(t,x_2)|&\leq \int_{0}^{t} \int_{|z| \geq 1} \int_{\mathbb{R}^{d}} K ( t-s, y )
  \Big[ f ( s,x_{1} -y +z ) -  f ( s,x_{1} -y )\\
    &\qquad\qquad\qquad\quad - f ( s,x_{2} -y +z ) + f ( s,x_{2} -y )\Big] dy d \mu ( z ) ds \\
    & \leq 2 t\|f \|_{C_{b,t}C_{b,x}^{1}} |x_{1} - x_{2} | \int_{|z| \geq 1} d \mu ( z ).
\end{align*}
Then since  and $\|I_2(t,\cdot)\|_{C_b}\leq
2t\|f\|_{C_{b,t}C_{b,x}^{1}}\int_{|z|\geq1}d\mu(z)$, 
\begin{align*}
   |I_2(t,x_1)-I_2(t,x_2)|  &\leq
   (2\|I_2(t,\cdot)\|_{C_b})^\beta|I_2(t,x_2)-I_2(t,x_2)|^{1-\beta}\\
   &\leq 4 t\|f \|_{C_{b,t}C_{b,x}^{1}} \int_{|z| \geq 1} d \mu ( z )|x_{1} - x_{2} |^{1-\beta}.
\end{align*}

By the fundamental theorem, Fubini, and a change of variables,
\begin{align*}
    I_1(t,x) &= \int_{0}^t \int_{|z| < 1 } \Big[\int_{\mathbb{R}^{d}}
      \int_{0}^{1} \nabla_{x} K ( t-s, y + \sigma z )- \nabla_{x} K (
      t-s,y )\Big]\cdot z f ( s,x-y ) d \sigma dy d \mu ( z ) ds, \\
    &= \int_{0}^{t} \int_{0}^{1} \int_{\mathbb{R}^{d}} \int_{|z|<1} \nabla_{x} K ( t-s, y ) \cdot z \Big[ f ( s,x-y+ \sigma z ) - f ( s,x-y ) \Big] d \mu ( z ) dy d \sigma ds.
\end{align*}
It follows that
\begin{align*}
    I_1(t,x_1)-I_1(t,x_2) &= \int_{0}^{t} \int_{0}^{1} \int_{\mathbb{R}^{d}}\nabla_{x} K ( t-s, y ) \cdot  \int_{|z|<1} z \Big[ f ( s,x_{1}-y+ \sigma z )\\&\qquad- f ( s,x_{2}-y+ \sigma z ) - \big(  f ( s,x_{1}-y ) - f ( s,x_{2}-y )\big)\Big] d \mu ( z ) dy d \sigma ds.
\end{align*}
Since
\begin{align*}
    &|f ( x_{1} + \sigma z ) - f ( x_{1} ) - f ( x_{2} + \sigma z ) - f ( x_{2} ) |^{1-\beta+\beta} \leq 2 \| f \|_{C_{b,t}C_{b,x}^{1}}^{1-\beta}|x_{1} - x_{2}|^{1-\beta}
  \|f\|_{C_{b,t}C_{b,x}^{1}}^{\beta} |\sigma z|^{\beta},
\end{align*}
we see by Theorem \ref{L_heat_kernel_estimate}
and \ref{L1} that
\begin{align*}
  &|I_1(t,x_1)-I_1(t,x_2)| \\ & \leq \int_{0}^{t} \int_{\mathbb{R}^{d}} |\nabla_{x} K ( t-s,y ) | dy ds\  2 \| f \|_{C_{b,t}C_{b,x}^{1}}^{1-\beta}|x_{1} - x_{2}|^{1-\beta} \|f\|_{C_{b,t}C_{b,x}^{1}}^{\beta} \int_{|z|<1}|z|^{\beta+1} d \mu ( z ) \\
    &  \leq \mathcal{K} \tfrac\sigma{\sigma-1}T^{\frac{\sigma-1}{\sigma}} \int_{|z|<1} |z|^{\beta+1} d \mu ( z ) \|f\|_{C_{b,t}C_{b,x}^{1}} |x_{1} - x_{2}|^{1-\beta}.
\end{align*}

Combining the above two estimates, we conclude that
\begin{align*}
    |\mathcal{L} [ \Phi ( f ) ] ( t,x_{1} ) - \mathcal{L} [ \Phi ( f ) ] ( t,x_{2} ) | \leq c \| f  \|_{C_{b,t}C_{b,x}^{1}} |x_{1}-x_{2}|^{1-\beta},
\end{align*}
with $c=\frac{\sigma}{\sigma-1}T^{\frac{\sigma-1}\sigma}\mathcal
K\int_{|z|<1}|z|^{1+\beta}d\mu(z)+4T\int_{|z| \geq 1} d \mu ( z )$. By part a),
$\partial_{t} \Phi ( f ) ( t,x )= f ( t,x )+ \mathcal{L} [ \Phi ( f )
] ( t,x )$. Since
$$|f ( t,x ) - f ( t,y ) | \leq (2\|f\|_{C_b})^{\beta}|f (
t,x ) - f ( t,y ) |^{1-\beta}\leq 2\| f  \|_{C_{b,t}C_{b,x}^{1}}
|x-y|^{1-\beta},$$ we then also get that
\begin{align*}
    | \partial_{t} \Phi [ f ] ( t,x_{1} ) - \partial_{t} \Phi [ f ] ( t,x_{2} ) | \leq (2+c)\|f\|_{C_{b,t}C_{b,x}^{1}} |x_{1}-x_{2}|^{1-\beta}.
\end{align*}

\noindent b) \ {\em Part 2:} Uniform continuity in time. First note that
\begin{align*}
    &\mathcal{L}  \Phi [ f ] ( t,x )  - \mathcal{L}  \Phi [ f ] ( s,x )  = \int_{0}^t \mathcal{L} K ( \tau, \cdot ) \ast f ( t - \tau, \cdot ) d \tau - \int_{0}^{s} \mathcal{L} K ( \tau, \cdot ) \ast f ( s-\tau, \cdot ) d \tau \\
    &= \int_{0}^{s} \mathcal{L} K ( \tau, \cdot ) \ast \big( f (  t-\tau, \cdot ) -  f ( s-\tau, \cdot )\big) d \tau + \int_{s}^{t} \mathcal{L} K ( \tau, \cdot ) \ast f ( t - \tau, \cdot ) d \tau.
\end{align*}
Now we do as
before: Split the $z$-domain in two parts, use the fundamental theorem
and a change of variables to get
\begin{align*}
& \mathcal{L} K ( \tau, \cdot ) \ast \big(f (  t-\tau, \cdot ) -  f ( s-\tau, \cdot )\big)  \\[0.2cm]
&=  \int_{0}^{1} \int_{\mathbb{R}^{d}} \int_{|z|<1} \nabla_{x} K ( \tau, x-y ) \cdot z \big[     f ( t- \tau,y+ \sigma z ) - f ( t - \tau ,y )  \\[0.2cm] 
& \qquad \qquad \qquad \qquad \qquad - f ( s- \tau, y+\sigma z ) + f ( s-\tau,y ) \big] d \mu ( z ) dy d \sigma . \\
& +  \int_{\mathbb{R}^{d}} \int_{|z| \geq 1} K ( \tau, x-y ) [ f ( t- \tau, y + z ) - f ( t- \tau, y ) \\
& \qquad \qquad \qquad \qquad \qquad - f ( s-\tau, y+z ) + f ( s- \tau , y)] d \mu ( z ) dy. 
\end{align*}
Then we apply the trick
\begin{align*}
&| f ( t- \tau,y+ \sigma z ) - f ( t - \tau ,y ) - f ( s- \tau,y+
  \sigma z ) + f ( s - \tau ,y )|\\
  &\leq  2\omega_{f} ( |t-s| )^{1-\beta}
  (\|f\|_{C_{b,t}C^1_{b,x}}|z|)^{\beta}\quad\text{or}\quad 4 \omega_{f} ( |t-s| )^{1-\beta}
  \|f\|_{C_{b}}^{\beta},
\end{align*}
and find using Theorem \ref{L_heat_kernel_estimate}
and \ref{L1} that
\begin{align*}
  &\Big|\int_{0}^{s} \mathcal{L} K ( \tau, \cdot ) \ast \big( f (  t-\tau, \cdot ) -  f ( s-\tau, \cdot )\big) d \tau \Big|\\
  &\leq \Big[\frac{\sigma}{\sigma-1}s^{\frac{\sigma-1}\sigma}\mathcal K\int_{|z|<1}|z|^{1+\beta}d\mu(z)+4s\int_{|z| \geq 1} d \mu ( z
  )\Big]\|f\|_{C_{b,t}C^1_{b,x}}^\beta\omega_{f} ( |t-s| )^{1-\beta}.
\end{align*}
In a similar way we find that
\begin{align*}
&\Big|\int_{s}^{t} \mathcal{L} K ( \tau, \cdot ) \ast f ( t - \tau, \cdot )
d \tau\Big|\\
&\leq \Big[2\frac{\sigma}{\sigma-1}(t^{\frac{\sigma-1}\sigma}-s^{\frac{\sigma-1}\sigma})\mathcal K\int_{|z|<1}|z|^{1+\beta}d\mu(z)+2(t-s)\int_{|z| \geq 1} d \mu ( z
  )\Big]\|f\|_{C_{b}}\\
&\leq c_1\|f\|_{C_{b}}|t-s|^{\frac{\sigma-1}\sigma}.
\end{align*}

Combining all above estimates leads to
\begin{align*}
&\Big|\mathcal{L}  \Phi [ f ] ( t,x )  - \mathcal{L}  \Phi [ f ] ( s,x
  ) \Big| \leq c  \|f\|_{C_{b,t}C^1_{b,x}}^\beta\omega_{f} ( |t-s|
  )^{1-\beta}+\tilde c\|f\|_{C_{b}}|t-s|^{\frac{\sigma-1}\sigma},
\end{align*}
where $c$ is defined above and in the Lemma and
$$\tilde c=2\frac{\sigma}{\sigma-1}\mathcal K\int_{|z|<1}|z|^{1+\beta}d\mu(z)\underset{{s,t\in[0,T]}}{\max}\tfrac{\big|t^{\frac{\sigma-1}\sigma}-s^{\frac{\sigma-1}\sigma}\big|}{|t-s|^{\frac{\sigma-1}\sigma}}+2T^{\frac1\sigma}\int_{|z| \geq 1} d \mu ( z
  ).$$
Note that $\tilde c$ is finite. Then since
\begin{align*}
    \partial_{t} \Phi [ f ] ( t,x ) - \partial_{t} \Phi [ f ] ( s,x ) = f ( t,x ) - f ( s,x ) + \mathcal{L}  \Phi [ f ] ( t,x )  - \mathcal{L}  \Phi [ f ] ( s,x ),
\end{align*}
and $|f ( t,x ) - f ( s,x ) | \leq (2\|f\|_{C_b})^{\beta}\omega_{f} (
|t-s| )^{1-\beta}$, the continuity estimate for $\partial_{t} \Phi [ f
]$ follows.
\medskip

\noindent c) The proof follows by writing
\begin{align*}
    \partial_{x_{i}} \Phi ( g ) ( t,x ) = \int_{0}^{t} \partial_{x_{i}} K ( \tau,z ) g ( t-\tau,x-z ) dz d \tau,
\end{align*}
and then directly compute the difference $| \partial_{x_{i}} \Phi ( g ) ( t,x ) - \partial_{x_{i}} \Phi ( g ) ( s,y ) |$.
\medskip

\noindent The proof is complete.


\bibliographystyle{plain}
\end{document}